\documentclass[12pt]{amsart}
\usepackage{amsmath}
\usepackage{amssymb}
\usepackage{amsfonts,euscript}
\usepackage{amscd}
\usepackage{latexsym}
\usepackage{graphicx}
\usepackage{xypic}
\usepackage[T1]{fontenc}
\usepackage{setspace,ifpdf}
\ifpdf
\usepackage{pdfsync}
\fi

\theoremstyle{plain} 
\newtheorem{theorem}{Theorem}[section]
\newtheorem*{theorem*}{Theorem}
\newtheorem{lemma}[theorem]{Lemma}
\newtheorem{proposition}[theorem]{Proposition}
\newtheorem{corollary}[theorem]{Corollary}

\theoremstyle{definition}
\newtheorem{definition}[theorem]{Definition}
\newtheorem{example}[theorem]{Example}

\newtheorem{remark}[theorem]{Remark}

\newcommand{\ep}{\epsilon}

\renewcommand{\dim}{\operatorname{dim}}

\newcommand{\Sym}{\operatorname{Sym}}
\newcommand{\End}{\operatorname{End}}
\newcommand{\Hom}{\operatorname{Hom}}
\newcommand{\RHom}{\operatorname{RHom}}
\newcommand{\Ext}{\operatorname{Ext}}
\newcommand{\du}{\operatorname{d}}

\newcommand{\C}{{\mathbb{C}}}
\newcommand{\Z}{{\mathbb{Z}}}
\newcommand{\Q}{{\mathbb{Q}}}
\renewcommand{\H}{{\mathbb{H}}}
\newcommand{\R}{{\mathbb{R}}}

\newcommand{\Ve}[1]{V_{#1}}
\newcommand{\Pro}[1]{P_{#1}}
\newcommand{\Si}[1]{L_{#1}}
\newcommand{\Tilt}[1]{T_{#1}}
\newcommand{\Pot}{\Phi_{12}}
\newcommand{\Pto}{\Phi_{2\bar{1}}}

\newcommand{\bigmid}{\hs\Big{|}\hs}
\newcommand{\subs}{\subseteq}

\renewcommand{\iff}{\Leftrightarrow}
\newcommand{\impl}{\Rightarrow}

\newcommand{\hs}{\hspace{3pt}}

\newcommand{\tPi}{\boldsymbol{\Pi}_\a^\bullet}

\newcommand{\M}{\mathfrak{M}}

\newcommand{\eX}{\EuScript{X}}

\newcommand{\vp}{\varphi}
\renewcommand{\a}{\alpha}
\renewcommand{\b}{\beta}

\newcommand{\de}{\delta}
\newcommand{\ga}{\gamma}

\newcommand{\cN}{\mathcal{N}}
\newcommand{\mfD}{\mathfrak{D}}

\newcommand{\qd}{!}
\newcommand{\gd}{\vee}

\newcommand{\cal}{\mathcal}

\newcommand{\cB}{{\cal B}}
\newcommand{\cC}{{\cal C}}

\newcommand{\cF}{{\cal F}}
\renewcommand{\cH}{{\cal H}}
\newcommand{\cI}{{\cal I}}

\newcommand{\cO}{{\cal O}}
\newcommand{\cP}{{\cal P}}

\newcommand{\cV}{{\cal V}}

\newcommand{\conv}{\star}
\newcommand{\ori}{\mathit{or}}

\newcommand{\la}{\langle}
\newcommand{\ra}{\rangle}
\newcommand{\lra}{\leftrightarrow}
\newcommand{\lrao}[1]{\overset{#1}\lra}
\newcommand{\wt}{\widetilde}

\newcommand{\bdy}{{\partial}}

\newcommand{\udot}{{\scriptscriptstyle \bullet}}

\newcommand{\ssm}{\smallsetminus}
\renewcommand{\la}{\lambda}

\renewcommand{\emptyset}{\varnothing}

\newcommand{\PI}{\mathsf{P}}

\newcommand{\PIS}{\mathsf{S}}

\newcommand{\Bas}{{\mathbb{B}}}

\newcommand{\Aeone}{{$\mathrm{A}_{\ext}$1}}
\newcommand{\Aetwo}{{$\mathrm{A}_{\ext}$2}}
\newcommand{\Aethree}{{$\mathrm{A}_{\ext}$3}}
\newcommand{\Aone}{{\rm A1}}
\newcommand{\Atwo}{{\rm A2}}
\newcommand{\Athree}{{\rm A3}}

\newcommand{\ecore}{\eX_\ext}

\newcommand{\ext}{{\mathrm{ext}}}

\begin{document}
\title{Gale duality and Koszul duality}
\author[T. Braden]{Tom Braden}
\author[A. Licata]{Anthony Licata}
\author[N. Proudfoot]{Nicholas Proudfoot}
\author[B. Webster]{Ben Webster}

\address{Department of Mathematics and Statistics,
University of Massachusetts} 
\email{braden@math.umass.edu}
\urladdr{http://www.math.umass.edu/$\sim$braden}

\address{Department of Mathematics, Stanford University} 
\email{amlicata@math.stanford.edu}
 \urladdr{http://math.stanford.edu/$\sim$amlicata}

\address{Department of Mathematics,
University of Oregon}
\email{njp@uoregon.edu}
\urladdr{http://www.uoregon.edu/$\sim$njp}

\address{Department of Mathematics, Massachusetts Institute of Technology} 
\email{bwebster@math.mit.edu}
\urladdr{http://math.mit.edu/$\sim$bwebster}

\begin{abstract}
  Given a hyperplane arrangement in an affine space equipped with a linear functional,
  we define two finite-dimensional, noncommutative algebras, both of which
  are motivated by the geometry of hypertoric varieties.
  We show that these algebras are Koszul dual to each other,
  and that the roles of the two algebras are reversed by Gale duality.
  We also study the centers and representation categories
  of our algebras, which are in many ways analogous to integral blocks of
  category $\mathcal O$.
  \end{abstract}

\begin{spacing}{1.2}
\maketitle
\section{Introduction}
\label{sec:introduction}

In this paper we define and study a class of finite-dimensional graded
algebras which are related to the combinatorics of hyperplane arrangements
and to the geometry of hypertoric varieties.
The categories of representations of these algebras are similar in
structure to the integral blocks of category $\cO$, originally introduced
by Bernstein-Gelfand-Gelfand \cite{BGG}.  Our categories share many
important properties with such blocks, including
a highest weight structure, a Koszul
grading, and a relationship with the geometry of a symplectic variety.
As with category $\cO$, there is a nice interpretation of Koszul duality in our setting,
and there are interesting families of functors between our categories.
In this paper we take a combinatorial approach, analogous to that
taken by Stroppel in her study of category $\cO$ \cite{Str03}.
In a subsequent paper \cite{BLPW2} we will take an approach to these
categories more analogous to the classical perspective on category
$\cO$; we will realize them as categories of modules over an infinite
dimensional algebra, and as a certain category of sheaves on a hypertoric
variety, related by a localization theorem extending
that of Beilinson-Bernstein \cite{BB}. 

\subsection{} To define our algebras, we take as our input what we call a
{\bf polarized arrangement} $\cV = (V, \eta,\xi)$, consisting of
a linear subspace $V$ of a coordinate vector space $\R^I$,
a vector $\eta\in\R^I/V$, and a covector $\xi\in V^*$.
It is convenient to think of these data as describing
an affine space $V_\eta \subs \R^I$ given by translating 
$V$ away from the origin by $\eta$, together with an affine linear
functional on $V_\eta$ given by $\xi$ and a finite 
hyperplane arrangement $\cH$ in $V_\eta$, whose
hyperplanes are 
the restrictions of the coordinate hyperplanes in $\R^I$.

If $\cV$ is {\bf rational}, meaning that $V$, $\eta$, and $\xi$ are
all defined over $\Q$, then we may associate to $\cV$ a
hyperk\"ahler orbifold $\M_\cH$ 
called a {\bf hypertoric variety}.
The hypertoric variety depends only on the arrangement $\cH$ (that is, on $V$ and $\eta$).  
It is defined as a hyperk\"ahler quotient of the
quaternionic vector space $\H^n$ by an $(n-\dim V)$-dimensional real torus
determined by $V$, where the quotient parameter is specified by
$\eta$.  By fixing one complex structure on $\M_\cH$ we obtain an algebraic symplectic variety
which carries a natural hamiltonian action
of an algebraic torus with Lie algebra $V^*_\C$, and $\xi$ determines
a one-dimensional subtorus.  The definitions and results of this paper
do not require any knowledge of hypertoric varieties (indeed, they
will hold even if $\cV$ is not rational, in which case there are no
varieties in the picture).  They will, however, be strongly motivated
by hypertoric geometry, and we will take every opportunity to point
out this motivation.  
The interested reader can learn more about
hypertoric varieties in the survey \cite{Pr07}.

Given a polarized arrangement $\cV$, we give combinatorial\footnote{Here and elsewhere in the paper
we use the term "combinatorial" loosely to refer to 
constructions involving finite operations on linear algebraic data.}  definitions
of two quadratic algebras, which we denote by $A = A(\cV)$ and $B = B(\cV)$.  
If $\cV$ is rational, both rings have geometric interpretations.
The $\C^*$-action on $\M_\cH$ given by $\xi$ determines a 
lagrangian subvariety $\eX \subset \M_\cH$, consisting of 
all points $x$ for which $\lim_{\lambda \to \infty} \lambda \cdot x$ exists.
In this case, $B$ is isomorphic to the direct sum
of the cohomology rings of all pairwise intersections of components
of $\eX$, equipped with a convolution product (Proposition \ref{geom b}).
On the other hand, we will show 
in a forthcoming paper \cite{BLPW2} that $A$ is the endomorphism algebra of a projective generator of 
a category of modules over a quantization
of the structure sheaf of $\M_\cH$ which are supported
on $\eX$.  This construction is motivated by geometric 
representation theory: the analogous category of modules 
on $T^*(G/B)$, the cotangent bundle of a flag variety, is equivalent to a regular block 
of category $\cO$ for the Lie algebra $\mathfrak g$.


\subsection{} There are two forms of duality lurking in this picture, one 
coming from combinatorics
and the other from ring theory.  We define the {\bf Gale dual} of a polarized
arrangement $\cV = (V,\eta,\xi)$ as the triple $$\cV^\gd = (V^\perp,
-\xi, -\eta),$$ where $V^\perp\subs(\R^I)^*$ is the space of linear
forms on $\R^I$ that vanish on $V$.  
On the other hand, to any quadratic algebra $E$ we may
associate its {\bf quadratic dual} algebra $E^\qd$.  
We show that the algebras $A$ and $B$ are
dual to each other in {\em both} of these senses:

\begin{theorem*}[A]
There are ring isomorphisms 
$A(\cV)^\qd \cong A(\cV^\gd)$ and $A(\cV^\gd) \cong B(\cV)$.
\end{theorem*}

We prove the following three facts about the rings $A$ and $B$, all of which are analogous
to results about category $\cO$ and the geometry
of the Springer resolution \cite{Spa76,Irv85,Bru06,Str06b,SW08}.

\vspace{\baselineskip}
\begin{theorem*}[B]~\begin{enumerate}
  \item The algebras $A$ and $B$ are quasi-hereditary and Koszul (and thus are Koszul dual).
  \item If $\cV$ is rational, then the center of $B$ is 
  canonically isomorphic to the cohomology ring of $\M_\cH$.
\item There is a canonical bijection between indecomposible
  projective-injective $B$-modules and compact chambers of the
  hyperplane arrangement $\cH$; if $\cV$ is rational, these are in bijection with
  the set of all irreducible projective lagrangian subvarieties of $\M_{\cH}$.
\end{enumerate}
\end{theorem*}

Part (2) of Theorem (B) is analogous to a result of \cite{Bru06,Str06b}, which 
says that the center of a regular block of parabolic category $\cO$ for 
$\mathfrak{g} = \mathfrak{sl}_n$ is isomorphic to the cohomology of a Springer fiber.
Note that the cohomology of the hypertoric variety $\M_\cH$ is independent
of \emph{both} parameters $\eta$ and $\xi$ \cite{Ko,HS,Pr07}.  This leads us to ask to what
extent the algebras themselves depend on $\eta$ and $\xi$.  It turns out that 
the algebras for polarized arrangements 
with the same underlying vector space $V$ may not be
isomorphic or Morita equivalent, but they are {\em derived} Morita equivalent.

\begin{theorem*}[C]
The bounded derived category of graded modules over $A(\cV)$ or $B(\cV)$
depends only on the subspace $V\subset\R^I$.
\end{theorem*}

The functors that realize these equivalences are analogues of twisting and shuffling 
functors on category $\cO$.  

\subsection{} The paper is structured as follows.  In Section \ref{sec:lp}, we lay out the 
combinatorics and linear algebra of polarized arrangements,
introducing definitions and constructions upon which we will rely throughout the
paper.  Section \ref{sec:quiver} is devoted to the algebra $A$, and contains a proof
of the first isomorphism of Theorem (A).
In Section \ref{sec:convolution} we turn to the algebra $B$; in it we complete the proof of Theorem
(A), as well as part (2) of Theorem (B).  Section \ref{sec:repr-categ} begins
with a general overview of quasi-hereditary and Koszul algebras, culminating in the
proofs of parts (1) and (3) of Theorem (B).  In Section \ref{derived equivalences}
we prove Theorem (C), and along the way we study Ringel duality, Serre functors,
and mutations of exceptional collections of $A$-modules.

Let $\cH^\vee$ be the hyperplane arrangement associated to $\cV^\gd$.
The relationship between the hypertoric varieties 
$\M_\cH$ and $\M_{\cH^\vee}$ implied by our results
is a special case of a duality relating pairs of symplectic algebraic varieties.
This duality, which we call {\bf symplectic duality}, will be 
explored in a more general context in future papers \cite{BLPW2, BLPW3}.  
Other examples of symplectic
dual pairs include Springer resolutions for Langlands dual groups and
certain pairs of moduli spaces of instantons on surfaces. These
examples all appear as the Higgs branches of the moduli space of vacua
for mirror dual 3-d $\cN=4$ super-conformal field theories, or as the Higgs and Coulomb branches of a single such theory. 
For hypertoric varieties, this was shown by Strassler and Kapustin \cite{KS99}.
We anticipate that our results on symplectic duality will ultimately
be related to the structure of these field theories.

\bigskip
\noindent{\em Acknowledgments}.  The authors would like to thank Jon
Brundan, Michael Falk, Davide Gaiotto, Sergei Gukov, Christopher Phan,
Catharina Stroppel, and Edward Witten for invaluable conversations.

T.B. was supported in part by NSA grant H98230-08-1-0097.  A.L. was supported in part by a Clay Liftoff Fellowship.
N.P. was supported in part by NSF grant DMS-0738335.
B.W. was supported by a Clay Liftoff Fellowship and
an NSF Postdoctoral Research Fellowship.
T.B. would like to thank the Institute for Advanced Studies of the
Hebrew University, Jerusalem, and Reed College for their hospitality. 

\section{Linear programming}\label{sec:lp}
\subsection{Polarized arrangements}
Let $I$ be a finite set.  
\begin{definition}
  A {\bf polarized arrangement} indexed by $I$ is a triple $\cV = (V, \eta, \xi)$
  consisting of
  \begin{itemize}
  \item a vector subspace $V \subset \R^I$,
  \item a vector $\eta \in \R^I/V$, and
  \item a covector $\xi \in V^*=(\R^I)^*/V^\perp$,
  \end{itemize}
such that
  \begin{enumerate}
 \item[(a)] every lift of $\eta$ to $\R^I$ has at least $|I|-\dim V$ non-zero entries, and
\item[(b)] every lift of $\xi$ to $(\R^I)^*$ has at least $\dim V$ non-zero entries.
\end{enumerate}
(Note that for $V$ fixed, a generic $\eta$ will satisfy (a), and a generic $\xi$ will
satisfy (b).)
If $V$, $\eta$, and $\xi$ are all defined over $\Q$, then $\cV$ is called
{\bf rational}.
\end{definition}

Associated to a (not necessarily rational) polarized arrangement 
$\cV = (V,\eta,\xi)$ is an arrangement $\cH$ of  $|I|$ hyperplanes 
in the affine space 
\[V_\eta = \{x \in \R^I \mid \eta = x + V\},\]
whose $i^\text{th}$
hyperplane is given by $$H_i = \{x \in V_\eta \mid x_i = 0\}.$$ 
Note that $H_i$ could be empty if $V$ is contained in the 
coordinate hyperplane $\{x_i = 0\}$.  In that case we refer to
$i$ as a \textbf{loop} of $\cV$, since it represents a loop 
in the matroid associated to $V$.

For any subset $S\subset I$, we let $$H_S = \bigcap_{i\in S} H_i$$
be the {\bf flat} spanned by the set $S$.
Condition (a) implies that $\cH$ is simple, meaning that 
$\operatorname{codim} H_S = |S|$
whenever $H_S$ is nonempty.  Observe that $\xi$ may be regarded as an 
affine-linear functional on $V_\eta;$ it does not have well-defined values,
but it may be used to compare any pair of points.
Condition (b) implies that $\xi$
is generic with respect to the arrangement,
in the sense that it is not constant on any positive-dimensional flat $H_S$.

\subsection{Boundedness and feasibility}\label{b and f}
Given a sign vector $\a \in \{\pm 1\}^I$, let
\[
	\Delta_\a = V_\eta \cap \{x \in \R^I \mid \a(i)x_i \ge 0 \;\text{for all $i$}\}
\] 
and 
$$\Sigma_\a = V \cap \{x \in \R^I \mid \a(i)x_i \ge 0 \;\text{for all $i$}\}.$$
If $\Delta_\a$ is nonempty, it is the closed chamber of the arrangement 
$\cH$ where the defining
equations of the hyperplanes are replaced by 
inequalities according to the signs in $\a$.
The cone $\Sigma_\a$ is the corresponding chamber of the central arrangement 
given by translating the hyperplanes of $\cH$ to the origin. 
It is always nonempty, as it contains $0$.  If $\Delta_\a$
is nonempty, then $\Sigma_\a$ is the recession cone of $\Delta_\a$ --- the set of direction vectors of rays 
contained in $\Delta_\a$ (see \cite[\S1.5]{Z}).  Note that $\Sigma_\a$ is independent of $\eta$, so even
if $\Delta_\a = \emptyset$, it is possible to change $\eta$
(in terms of $\cH$, this corresponds to translating the hyperplanes)
to obtain a nonempty $\Delta_\a$, and 
then take its cone of unbounded directions.

We now define subsets $\cF, \cB, \cP \subset\{\pm 1\}^I$ as follows.  First we let
\[
	\cF = \{\a \in \{\pm 1\}^I \mid \Delta_\a \ne \emptyset\}.
\]
Elements of $\cF$ are called {\bf feasible}.  It is clear that $\cF$ depends only
on $V$ and $\eta$.
Next, we let $$\cB = \{\a \in \{\pm 1\}^I \mid \xi(\Sigma_\a)\,\text{is bounded above}\}.$$
Elements of $\cB$ are called {\bf bounded}, and it is clear that $\cB$ depends
only on $V$ and $\xi$.
Elements of the intersection
$$\cP:= \cF\cap \cB = \{\a \in \{\pm 1\}^I \mid \xi(\Delta_\a)\,\text{is 
nonempty and bounded above}\}$$ are called {\bf bounded feasible};
here $\xi(\Delta_\a)$ is regarded as a subset of the affine line.
Our use of these terms comes from linear programming, where we consider
$\alpha$ as representing the linear program ``find the maximum of
$\xi$ on the polyhedron $\Delta_\alpha$''.  

\begin{example}\label{bounded and feasible example}
Let $I = \{1, 2, 3, 4\}$, and let $\cV = (V, \eta,\xi)$
be the polarized arrangement where 
\[V = \{(y, y, x, x+y) \mid (x, y) \in \R^2\},\]
$\eta$ is the image of $(-1, 0, 0, -2)$ in $\R^4/V$, and 
$\xi$ is the image of 
$(2, 0, 1, 0) \in (\R^4)^*$ in $V^* = (\R^4)^*/V^\bot$.
In terms of the $(x, y)$ coordinates on $V_\eta$, 
the inequalities defining the positive sides of the 
four hyperplanes are $y \ge 1$, $y \ge 0$, $x \ge 0$, and
$x + y \ge 2$.  The functional $\xi$, up to an additive constant, 
is $\xi(x, y) = x + 2y$.  See Figure \ref{first figure}, where
we label all of the feasible regions with the appropriate 
sign vectors.  The bounded feasible regions are shaded.
Note that besides the five unbounded feasible regions
pictured, there is one more unbounded sign vector,
namely $+-++$, which is infeasible.

\begin{figure}[htb]
\includegraphics[totalheight=2.6in]{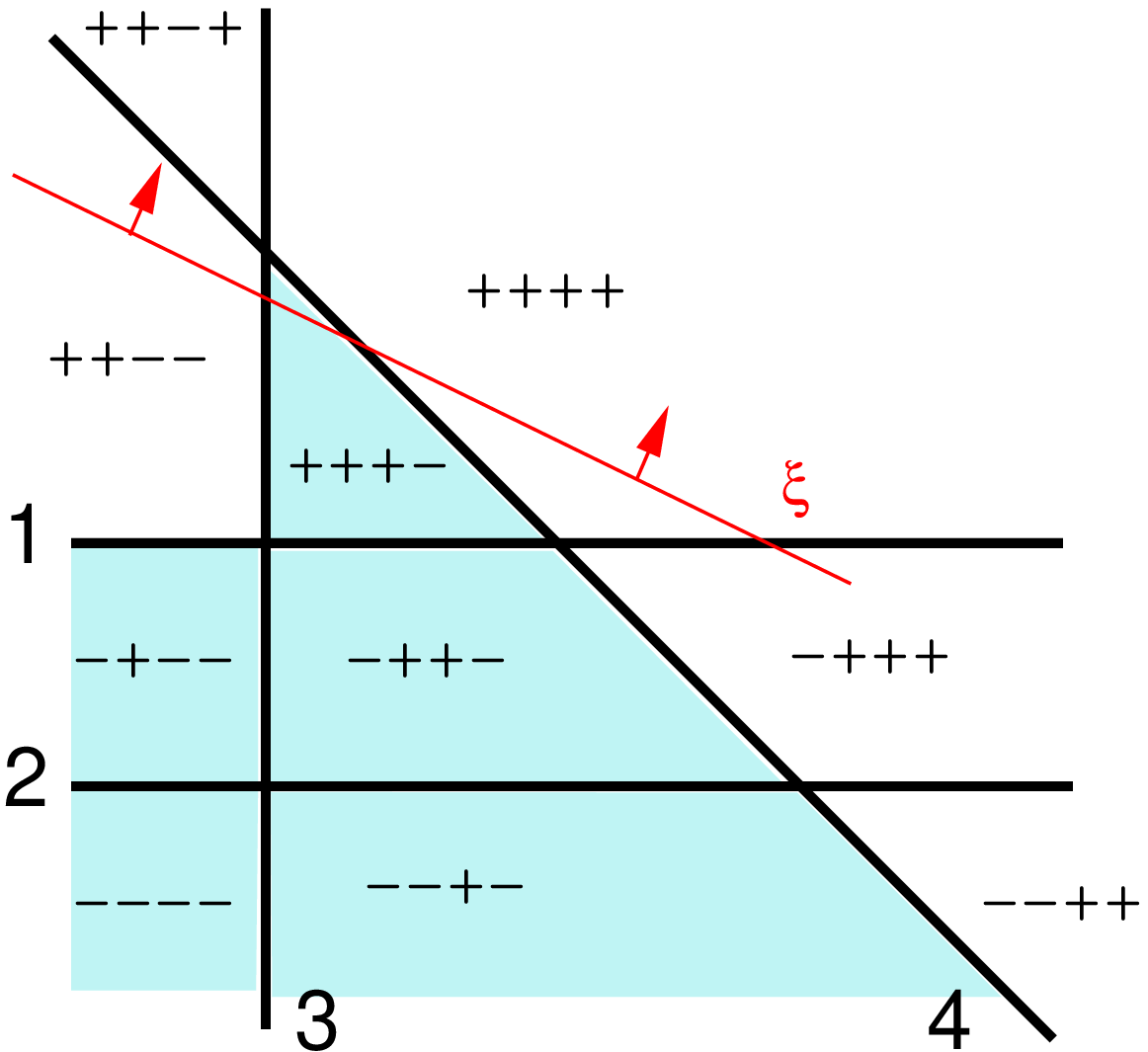}
\caption{Example of bounded and feasible chambers}\label{first figure}
\end{figure}

\end{example}

\subsection{Gale duality}
Here we introduce one of the two main dualities of this paper.
\begin{definition}
The {\bf Gale dual} $\cV^\gd$ of a polarized arrangement 
$\cV=(V,\eta,\xi)$ is given by the triple $(V^\perp, -\xi, -\eta)$.
We denote by $\cF^\gd$, $\cB^\gd$, and $\cP^\gd$ the
feasible, bounded, and bounded feasible sign vectors for $\cV^\gd$,
and we denote by $V^\bot_{-\xi}$ the affine space for 
the corresponding hyperplane arrangement $\cH^\vee$.
\end{definition}

This definition agrees with the notion of duality in linear programming:
the linear programs for $\cV$ and $\cV^\gd$ and a fixed sign vector
$\alpha$ are dual to each other.  The following key result is 
a form of the strong duality theorem of linear programming.

\begin{theorem}\label{ftlp}
$\cF^\gd = \cB$, $\,\cB^\gd = \cF$, and therefore $\cP^\gd = \cP$.
\end{theorem}
\begin{proof}
It is enough to show that $\alpha = (+1, \dots, +1)$ is feasible 
for $\cV$ if and only if it is bounded for $\cV^\gd$.
The Farkas lemma \cite[1.8]{Z} says that exactly one of the following statements
holds: 
\begin{enumerate}
 \item there exists a lift of $\eta$ to $\R^I$ which lies in 
$\R^I_{\ge 0}$,
\item there exists $c \in V^\perp\subset(\R^I)^*$ 
which is positive on $\R^I_{\ge 0}$ and negative on $\eta$.
\end{enumerate}
Statement (1) is equivalent to $\alpha \in \cF$,
while a vector $c$ satisfying (2) lies in $\Sigma^\vee_\alpha$, so $c(\eta) < 0$ means that 
$-\eta$ is not bounded above on $\Sigma^\vee_\alpha$.
\end{proof}

\begin{example}\label{Gale example}
Continuing with polarized arrangement $\cV$ of Example 
\ref{bounded and feasible example}, we have
\[V^\gd = V^\perp = \{(X+Y, -X, Y, -Y) \mid (X, Y) \in \R^2\}.\]
So in $(X, Y)$ coordinates, 
the inequalities defining the positive sides of the four 
hyperplanes are, in order,  $X + Y \ge 2$, $X \le 0$, $Y \ge 1$, $Y \le 0$
(the fact that these are the same as for
$\cV$ up to sign and reordering is a coincidence).  
The covector $\xi^\vee = -\eta$ gives
the function $(X, Y) \mapsto X - Y$.  Figure \ref{second figure}
shows the feasible and bounded feasible regions for $\cV^\gd$.

\begin{figure}[htb]
 \includegraphics[totalheight=2.6in]{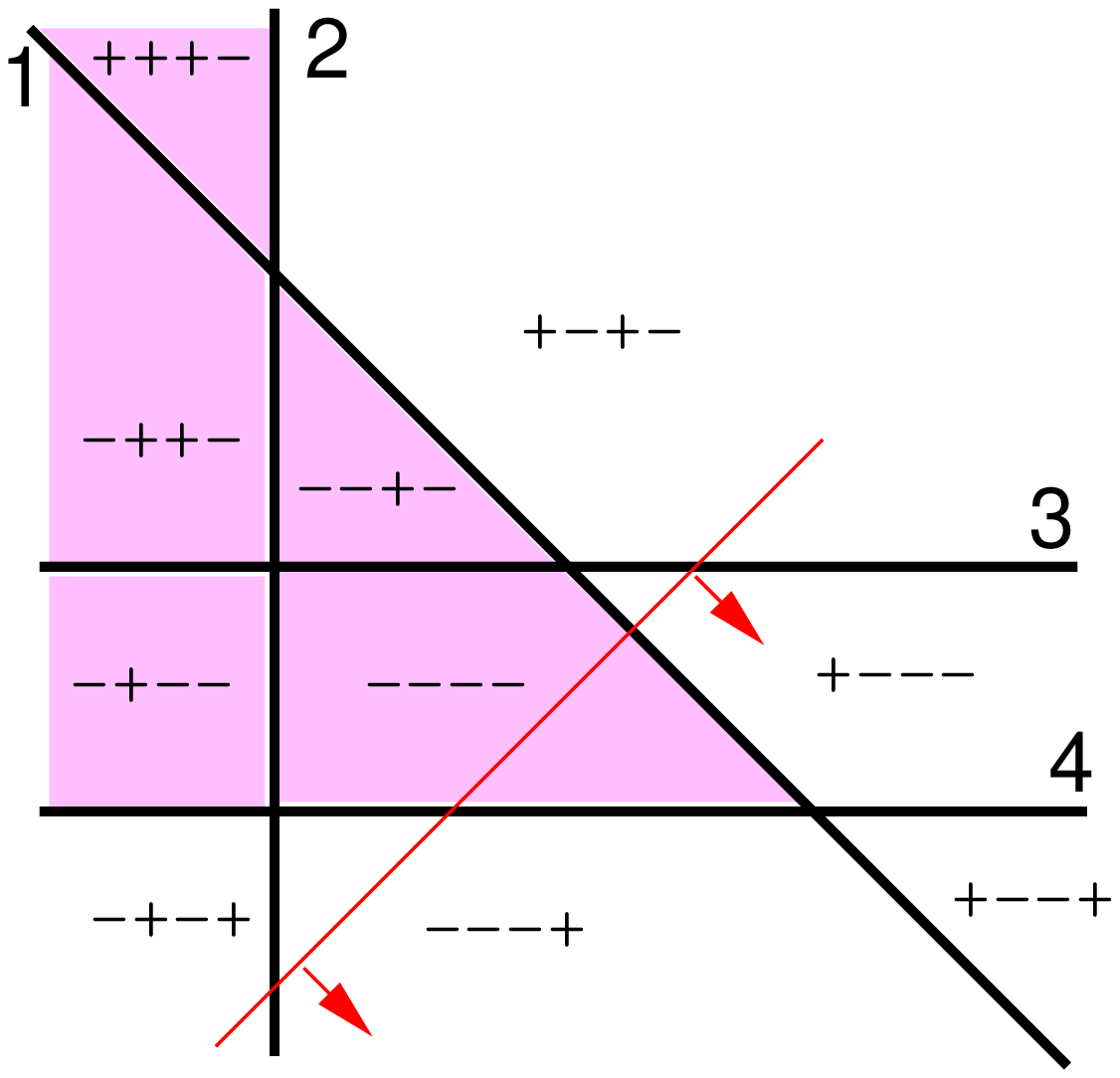}
\caption{Bounded and feasible chambers for the Gale dual arrangement}\label{second figure}
\end{figure}

\end{example}

\subsection{Restriction and deletion}  

We define two operations which reduce the number of hyperplanes in a polarized 
arrangement as follows.  First, consider a 
subset $S$ of $I$ such that $V + \R^{I\ssm S} = \R^I$.
Since $\eta$ is assumed to be generic, this condition is equivalent to saying 
that $H_S \ne \emptyset$.  
Consider the natural isomorphism
\[
	i: \R^{I\ssm S}/(V\cap\R^{I\ssm S}) \longrightarrow \R^I/V
\]
induced by the inclusion of $\R^{I\ssm S}$ into $\R^I$.
We
define a new polarized arrangement $\cV^S = (V^S,\xi^S,\eta^S)$, 
indexed by the set $I\ssm S$, as follows:
\begin{align*}
  V^S &:= V \cap \R^{I \ssm S}\subset \R^{I\ssm S}\\ 
  \xi^S &:= \xi|_{V^S}\\ \eta^S &:= i^{-1}(\eta).
\end{align*}
The arrangement $\cV^S$ is called the
{\bf restriction} of $\cV$ to $S$, since the associated hyperplane
arrangement is isomorphic to the
hyperplane arrangement obtained by restricting to the subspace $H_{S}$.

Dually, suppose that $S \subset I$ is a subset such that $\R^S \cap V = \{0\}$, and 
let
$$\pi \colon \R^I \to \R^{I\ssm S}$$ be the coordinate projection, which 
restricts to an isomorphism 
\[
	\pi|_V: V \longrightarrow \pi(V).
\]
We define another polarized arrangement $\cV_S = (V_S,\eta_S,\xi_S)$, 
also indexed by $I \ssm S$,  as follows:
\begin{align*}
  V_S &= \pi(V) \subset \R^{I\ssm S}\\
  \eta_S &= \pi(\eta)\\
  \xi_S &= \xi \circ \pi|_V^{-1}.
\end{align*}
The arrangement $\cV_S$ is called the {\bf deletion} of $S$ from $\cV$, 
since the associated hyperplane arrangement
is obtained by removing the hyperplanes $\{H_i\}_{i\in S}$ from 
the arrangement associated to $\cV$.
The following lemma is an easy consequence of the definitions.

\begin{lemma} \label{Gale dual system}
$(\cV^S)^\gd$ is equal to $(\cV^\gd)_S$.
\end{lemma}

\begin{example} We continue with 
Examples \ref{bounded and feasible example} and \ref{Gale example}.
The deletion $\cV_{\{3, 4\}}$ is not defined, since
$\R^{\{3, 4\}} \cap V^\bot = \R \cdot (0, 0, -1, 1) \ne 0$;
this can be seen in Figure \ref{first figure} because removing
the hyperplanes $3$ and $4$ leaves hyperplanes whose normal
vectors do not span.  Dually, the restriction $(\cV^\gd)^{\{3, 4\}}$ is
not defined, since $V^\bot + \R^{\{1, 2\}} \ne \R^4$.  This can be 
seen geometrically because $H_{\{3, 4\}} = V^\bot_{- \xi} \cap \R^{\{1, 2\}}$
is empty.

On the other hand, we can form the deletion $\cV_{\{4\}}$; it gives
an arrangement of three lines in a plane, with the first two 
parallel.  On the dual side, $(\cV^\gd)^{\{4\}}$ is an arrangement
of two points on a line; the third hyperplane is a loop or 
empty hyperplane, since in the original arrangement $\cV^\gd$
the hyperplanes $H_3$ and $H_4$ are parallel.
\end{example}

\subsection{The adjacency relation}

Define a relation $\lra$ on $\{\pm 1\}^I$ by 
saying
 \[
 	\a \lra \b
\]
if and only if $\a$ and $\b$ differ in exactly one entry; if $\a,\b \in \cF$, this means that
 $\Delta_\a$ and $\Delta_\b$ are obtained from each other by flipping across a single hyperplane.  
We will write $\a\lrao i\b$ to indicate that $\a$ and $\b$ differ in the $i^\text{th}$ component of 
$\{\pm 1\}^I$.  We will also denote this by $\beta = \alpha^i$.  
The following lemma says that an infeasible neighbor of a bounded
feasible sign vector is feasible for the Gale dual system.

\begin{lemma} \label{duality lemma}
 Suppose that $\a \notin \cF$, $\b \in \cP$, and $\a \lra \b$.  
Then $\a \in \cF^\gd$.
\end{lemma}

\begin{proof}
Suppose that $\a \lrao i\b$.
The fact that $\b\in\cP\subs\cF$ tells us that $\Delta_\b\neq\emptyset$,
while $\Delta_\a = \emptyset$, thus
$$H_i \cap \Delta_\b = \Delta_\a\cap \Delta_\b = \emptyset.$$
From this we can conclude that
$$\Sigma_\a = \Sigma_\b\cap\{x_i=0\}\subset\Sigma_\b.$$
The fact that $\b\in\cP\subset\cB$, tells us that $\xi(\Sigma_\b)$ is bounded above,
thus so is $\xi(\Sigma_\a)$.
This in turn tells us that $\a\in\cB$, which is equal to $\cF^\gd$ by Theorem \ref{ftlp}.
\end{proof}

\subsection{Bases and the partial order}\label{sec:partial order}
Let $\Bas$ be the set of subsets
$b\subset I$ of order $\dim V$ such that $H_b\neq \emptyset$.
Such a subset is called a {\bf basis} for the matroid associated to $\cV$,
and in fact this property depends only on the subspace $V\subset \R^I$.
A set $b$ is a basis if and only if the composition $V\hookrightarrow\R^I\twoheadrightarrow\R^b$
is an isomorphism, which is equivalent to saying that
we have a direct sum decomposition $\R^I = V \oplus \R^{b^c}$,
where we put $b^c = I\ssm b$.
We have a bijection $$\mu:\Bas\to\cP$$ taking $b$ to the 
unique sign vector $\a$ such that $\xi$ attains its maximum on 
$\Delta_\a$ at the point $H_b$.

The covector $\xi$ induces a partial order $\le$ on $\Bas \cong \cP$.  It is 
the transitive closure of the relation $\preceq$, where $b_1 \prec b_2$
if $|b_1 \cap b_2| = |b_1| - 1 = \dim V - 1$ and $\xi(H_{b_1}) < \xi(H_{b_2})$.
The first condition means that $H_{b_1}$ and $H_{b_2}$ 
lie on the same one-dimensional flat, so $\xi$ cannot take the same 
value on these two points.

Let $\Bas^\gd$ denote the set of bases of $\cV^\gd$.
We have a bijection $b \mapsto b^c$ from $\Bas$ to $\Bas^\gd$,
since $\R^I = V \oplus \R^{b^c}$ if and only if
$\R^I = V^\bot \oplus \R^b$.
The next result says that this bijection is compatible with
the equality $\cP = \cP^\gd$ and the bijections 
$\mu\colon \Bas\to \cP$ and $\mu^\gd\colon \Bas^\gd\to \cP^\gd$.

\begin{lemma}\label{complement}
For all $b\in\Bas$, $\mu(b) = \mu^\vee(b^c)$.
\end{lemma}

\begin{proof}
Let $b \in \Bas$.  It will be enough to show that
$\mu(b) = \a$ if and only if 
\begin{enumerate}
\item the projection of $\a$ to $\{\pm 1\}^{b^c}$ is feasible 
for the restriction $\cV^b$, and
\item the projection of $\alpha$ to 
$\{\pm 1\}^{b}$ is bounded for the deletion $\cV_{b^c}$,
\end{enumerate} 
since Theorem \ref{ftlp} and Lemma \ref{Gale dual system}
tell us that these conditions are interchanged by dualizing and
swapping $b$ with $b^c$.

Note that $\cV^b$ represents the restriction of the 
arrangement to $H_b$, which is a point.  All of 
the remaining hyperplanes are therefore loops,
whose positive and negative sides are either 
all of $H_b$ or empty.  Condition (1) just says then
that $H_b$ lies in the chamber $\Delta_\a$.
Given that, in order for $H_b$ to be the $\xi$-maximum on 
$\Delta_\a$, it is enough that when all the hyperplanes not 
passing through $H_b$ are removed from the arrangement,
the chamber containing $\alpha$ is bounded.
But this is exactly the statement of condition (2),
which completes the proof.
\end{proof}

The following lemma demonstrates that this bijection is compatible with
our partial order.

\begin{lemma}\label{op}
Under the bijection $b \mapsto b^c$, the partial order on $\Bas^\gd$ is the opposite of
the partial order on $\Bas$.  That is, $b_1 \leq b_2$ if and only if $b_1^c\geq b_2^c$.
\end{lemma}

\begin{proof}
It will be enough to show that the generating relations $\prec$ are 
reversed.  So take bases $b_1$, $b_2$ with $|b_1 \cap b_2| = |b_1| - 1$.
Then there exist $i_1\in b_1$
and $i_2 \in b_2$ such that $b_2 = b_1\cup\{i_2\}\ssm\{i_1\}$.
We need to show that $\xi(H_{b_1}) < \xi(H_{b_2})$ holds
if and only if $\xi^\gd(H^\gd_{b^c_1}) > \xi^\gd(H^\gd_{b^c_2})$.
This reduces to the (easy) case where $|I|= 2$ by
replacing $\cV$ with the polarized
arrangement $\cV^{b_1\cap b_2}_{(b_1\cup b_2)^c}$ 
obtained by restricting 
to the one-dimensional flat spanned by $H_{b_1}$ and
$H_{b_2}$ and then deleting all hyperplanes but $H_{i_1}$ and $H_{i_2}$.
\end{proof}

For $b\in \Bas$, define
$$\cB_b = \{\a\in\{\pm 1\}^I\mid \a(i) = \mu(b)(i)\,\,\text{for all $i\in b$}\}.$$
Note that $\cB_b\subset\cB$, and
$\cB_b$ depends only on $V$ and $\xi$.  Geometrically, the feasible sign vectors in $\cB_b$ are those such that $\Delta_\a$ lies in the ``negative cone'' defined by $\xi$ with vertex $H_b$.  
Dually, we define 
$$\cF_b= \{\a\in\{\pm 1\}^I\mid \a(i)=\mu(b)(i)\,\,\text{for all $i\notin b$}\}=\cB^\gd_{b^c}$$
to be the set of sign vectors such that $H_b\in\Delta_\a$.  In particular, $\cF_b\subset\cF$.
We will need the following lemma in Section \ref{sec:repr-categ}.

\begin{lemma}\label{order lemma}
If $\mu(a)\in\cB_b$, then $a\leq b$.
\end{lemma}

\begin{proof}
Let $C$ be the negative cone of $\mu(b)$.
Then $\mu(a)\in\cB_b$ means that $H_a\in C$.
Let $C'$ be the smallest face of $C$ on which $H_a$ lives.
We will prove the lemma by induction on $d = \dim C'$.  If $d = 0$, then 
$a = b$ and we are done.  Otherwise, there is a one-dimensional
flat which is contained in $C'$ and passes through $H_{a}$.  
Following it in the $\xi$-positive direction, it must leave
$C'$, since $C'$ is $\xi$-bounded.  The point where it 
exits will be $H_{c}$ for some basis $c$, and $a < c$
by construction.  But $H_c$ lies on a face of $C$ of smaller 
dimension, so the inductive hypothesis gives $c \le b$.
\end{proof}

\begin{example}\label{partial order example}
We continue with Examples \ref{bounded and feasible example} 
and \ref{Gale example}.
Figure \ref{partial order figure} gives Hasse 
diagrams for the partial orders on 
$\Bas$ and $\Bas^\gd$.  


\begin{figure}[h]
 \includegraphics[totalheight=1.8in]{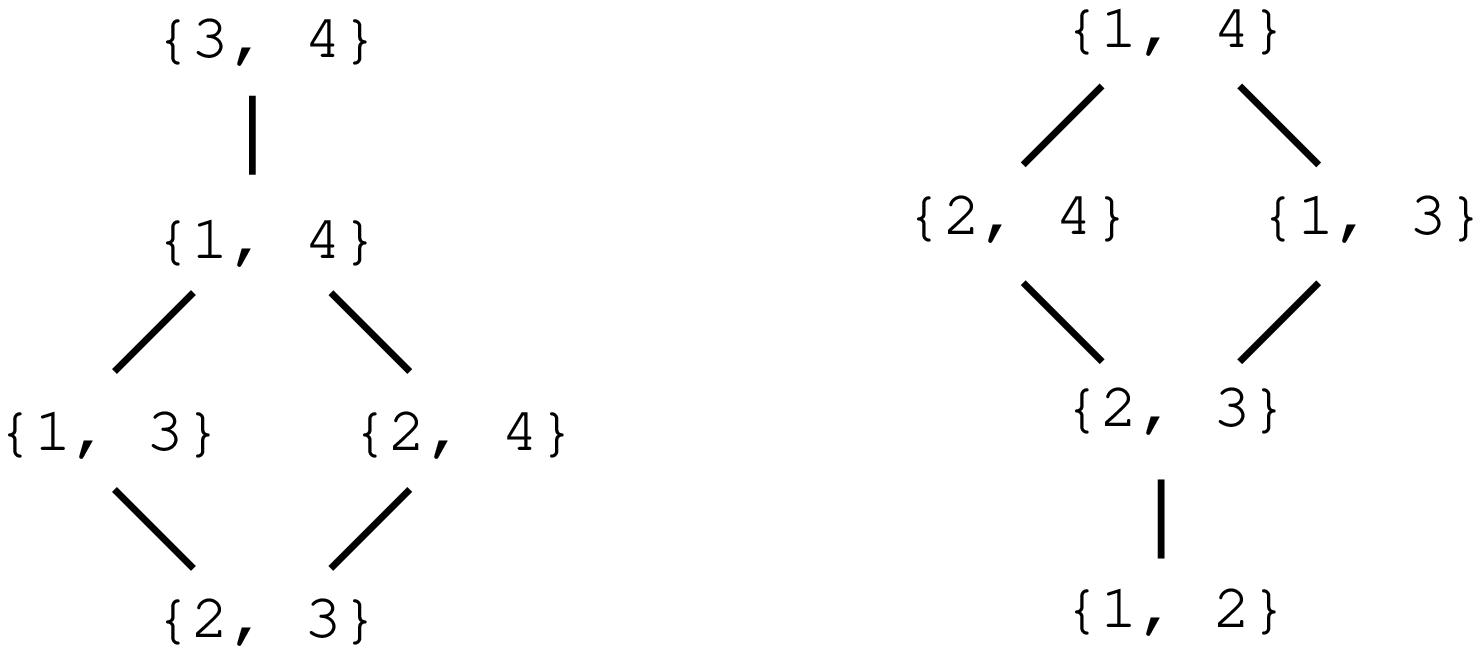}
\caption{Partial order on $\Bas$ and $\Bas^\gd$}\label{partial order figure}
\end{figure}

\end{example}

\section{The algebra $A$}\label{sec:quiver}
\subsection{Definition of \boldmath{$A$}}
Fix a polarized arrangement $\cV$, and 
consider the quiver $Q = Q(\cV)$ with vertex set $\cF$
and arrows $\{(\a,\b)\mid \a\lra\b\}$.
Note in particular that there is an arrow 
from $\alpha$ to $\beta$ if and only if there is an arrow from $\beta$ to $\alpha$.
Let $P(Q)$ be the algebra of real linear combinations of paths in the quiver $Q$,
generated by pairwise orthogonal idempotents $\{e_\alpha\mid\a\in\cF\}$
along with edge paths $\{p(\alpha,\beta)\mid\a\lra\b\}$.  
We use the following notation: if 
$\alpha_1 \lra \alpha_2 \lra \dots \lra \alpha_k$
is a path in the quiver, then we write\footnote{Note that with this convention,
	a representation of $Q$ is a {\em right} module over $P(Q)$.}
\[
	p(\alpha_1,\alpha_2,\dots,\alpha_k) := 
	p({\alpha_1,\alpha_2})\cdot p({\alpha_2, \alpha_3}) \cdot\ldots\cdot
 	p({\alpha_{k-1}, \alpha_k}).
\]
Let 
$t_i\in V^*$ be the restriction to $V$ of the $i^\text{th}$ coordinate function on $\R^I$.

\begin{definition}\label{defining A}
We define $A = A(\cV)$ to be the quotient of $P(Q)\otimes_\R \Sym V^*$ 
by the two-sided ideal generated by the following relations:
\begin{enumerate}
\item[\Aone:] If $\alpha\in \cF\ssm\cP$, then $e_\alpha=0$.\\
\item[\Atwo:] If four {\em distinct} elements $\alpha,\beta,\gamma,\delta \in \cF$
satisfy $\alpha \lra \beta \lra \gamma \lra \delta \lra \alpha$,
then 
\[p(\alpha, \beta, \gamma) = p(\alpha,\delta,\gamma).\]
\item[\Athree:] If $\a, \b \in \cF$ and $\a\lrao i\b$,
 then
$$p(\alpha,\beta,\alpha) = t_ie_\alpha.$$
\end{enumerate}
We put a grading on this algebra by letting $\deg(e_\a) = 0$ for all $\a\in\cF$, 
$\deg(p(\a,\b)) = 1$ for all
$\a \lra \b$, and $\deg(t_i) = 2$ for all $i\in I$.
\end{definition}

Let $Q_\cP$ be the subquiver of $Q$ consisting of the vertices 
that lie in $\cP$ and all arrows between them.
The following lemma tells us that $A$ is {\bf quadratic},
meaning that $A$ is generated over $A_0$ by $A_1$, and that the
only nontrivial relations are in degree 2.

\begin{lemma}\label{smaller quiver}
 The natural map $P(Q_\cP) \to A$ is surjective, and the kernel is 
generated in degree two.
\end{lemma}

\begin{proof}
For any $\a \in \cF$, define the set
\[I_\a := \{i\in I\mid \Delta_\a\cap H_i\neq\emptyset\} = \{i \in I\mid \a^i \in \cF\};\]
it indexes the codimension one faces of $\Delta_\a$.
Then the $t_i$ for $i\in I_\a$ are normal vectors to those
faces, and therefore span $V^*$.  Thus for any
$w \in V^*$, the corresponding element $w = \sum_{\alpha \in \cF} we_\a$
of $A$ can be written in terms of paths $p(\a, \b, \a)$, 
using the relation \Athree.  So 
$A$ is a quotient of the full path algebra $P(Q)$, and 
then by the relation \Aone\  
it is also quotient of $P(Q_\cP)$.  The 
relations for this presentation are generated 
by those of \Atwo\ (where the right side is understood to be $0$ if $\delta \notin \cP$)
along with linear relations among the various $p(\a, \b, \a)$
coming from the relations among the covectors $t_i \in V^*$. 
\end{proof}

\begin{remark}
The upshot of Lemma \ref{smaller quiver} is that we could have given a presentation
of $A$ that was more efficient than the one given in Definition \ref{defining A},
in the sense that it would have used fewer generators, all in degrees 0 and 1.
The trade-off would have been that the relations A2 and A3 would each have needed
to be split into cases, depending on whether or not $\b$ is bounded.
Furthermore, the map from $\Sym V^*$ to $A$ would have been less apparent
in this picture. 
\end{remark}

\begin{remark}\label{D-mod}
  In a subsequent paper \cite{BLPW2}
  we will show that, when $\cV$ is rational, the category of right $A$-modules is equivalent to
  a category of modules over a quantization of the structure sheaf of
  the hypertoric variety $\M_\cH$.
  In the special case where $\M_\cH \cong T^*X$ is the cotangent bundle of a projective variety $X$,
  these are just $\mathcal D$-modules on $X$ (microlocalized to $T^*X$) whose characteristic
  varieties are contained in the conormal variety to a certain
  stratification of the base.  More generally,
  the sheaves will be supported on the relative core of $\M_\cH$ (see Section \ref{toric}), a lagrangian subvariety
  of $\M_\cH$ defined by $\xi$.
\end{remark}

\begin{remark}\label{deligne}
The quiver $Q$ has appeared in the literature before; it is the Cayley graph
of the {\bf Deligne groupoid} of $\cH$ \cite{Par00}. 
Indeed, let $\tilde A(\cV)$ be the quotient of $P(Q)\otimes_\R\Sym(\R^I)^*$
by the relations \Aone, \Atwo, and $\tilde{\rm A}3$, where $\tilde{\rm A}3$ is obtained from \Athree\
by replacing $t_i\in V^*$ with the $i^\text{th}$ coordinate function on $\R^I$.
Let the {\bf Deligne semi-groupoid}
be the groupoid generated by paths in $Q$.  Then $\tilde A(\cV)$ is the quotient by
$\{e_\a\mid \a\in\cF\smallsetminus\cP\}$ of the semi-groupoid algebra of the Deligne
semi-groupoid, and $$ A\cong \tilde A\otimes_{\Sym(\R^I)^*}\Sym V^*
\cong A\otimes_{\Sym V^\perp}\R.$$
\end{remark}

\subsection{Taut paths}
We next establish a series of results that allow us to understand the elements
of $A$ more explicitly.  Though we do not need these results for the remainder
of Section \ref{sec:quiver}, they will be used in Sections \ref{sec:convolution} and 
\ref{sec:repr-categ}.

\begin{definition}\label{recursion statistic}
For a sequence $\alpha_1,\dots,\alpha_k$ of elements of $\{\pm 1\}^I$ and
an index $i \in I$, define
\[\theta_i(\alpha_1,\dots,\alpha_k) = \big | \{1 < j < k \mid \alpha_j(i) \ne \alpha_{j+1}(i) = \alpha_1(i)\}\big |.\]
This counts the number of times the sequence crosses the $i^\text{th}$ hyperplane
and returns to the original side.
\end{definition}

\begin{definition}
We say that a path $\a_1 \lra \a_2 \lra \dots \lra \a_k$ in $Q$ is 
{\bf taut} if it has minimal length among all paths from $\a_1$ to 
$\a_k$.  This is equivalent to saying that
the sign vectors $\a_1$ and $\a_k$ differ in exactly $k-1$ entries.
\end{definition}

\begin{proposition}\label{standard form}
 Let $\a_1 \lra \a_2 \lra \dots \lra \a_k$ be a path in $Q$. 
Then there is a taut path
\[
	\a_1 = \b_1 \lra \b_2 \lra \dots \lra \b_d = \a_k
\]
such that
\[
	p(\a_1, \dots, \a_k) = p(\b_1, \dots, \b_d)\cdot \prod_{i\in I} t_i^{\theta_i}
\]
in the algebra $A$, where $\theta_i = \theta_i(\a_1, \a_2,\dots, \a_k)$.
\end{proposition}

\begin{proof}
We can represent paths in the quiver geometrically by 
topological paths in the affine space $V_\eta$ in which $\cH$ lives.
Let $\phi\colon [0,1] \to V_\eta$ be a piecewise linear path
with the property that for any $t \in [0,1]$, the point $\phi(t)$ lies in at most 
one hyperplane $H_i$, and the endpoints $\phi(0)$ and $\phi(1)$ lie 
in no hyperplanes.  Such a path determines an element
$p(\phi) = p(\gamma_1, \dots, \gamma_r)$ in the algebra $A$, where
$\Delta_{\gamma_1}, \dots, \Delta_{\gamma_r}$ are the successive 
chambers visited by $\phi$.  We define $\theta(\phi) = \theta(\ga_1,\dots,\ga_r)$.

To represent our given element $p(\a_1,\dots, \a_k)$
as $p(\phi)$ for a path $\phi$, we choose points 
$x_1, \dots, x_k$ with $x_j \in \Delta_{\a_j}$
and let $\phi$ be the concatenation of the line segments
$\overline{x_jx_{j+1}}$.  By choosing the points $x_j$ generically 
we can assume that for every $1 < j < k$  
the line segment $\overline{x_1x_j}$
only passes through one hyperplane at a time, 
the plane containing $x_1$, $x_j$, and $x_{j+1}$ contains no point which lies 
in more than two hyperplanes, and any line through $x_{j+1}$ contained in this 
plane contains at most one point which is in two hyperplanes.

Given such points $x_i$, we construct a piecewise linear homotopy $\phi_t$ from $\phi = \phi_0$ to 
a straight-line path $\phi_1$ from $x_1$ to $x_k$, by contracting
the points $x_2, \dots, x_{k-1}$ one at a time 
along a straight line segment to $x_1$.  The 
sequence of chambers visited by the path changes only a
finite number of times, and at each step it can change in two 
possible ways.  First, when the line segment $\overline{x_jx_{j+1}}$ 
passes through the intersection of two hyperplanes, a sequence
$\a \lra \b \lra \gamma $ with $\a \ne \gamma$ 
is replaced by $\a \lra \delta \lra \gamma$ with $\delta \ne \b$.
The quiver relation \Atwo\  implies that the corresponding element in the algebra $A$
does not change.

Second, when the point $x_j$ passes through a hyperplane $H_i$, either
the sequence of chambers visited by the path is left unchanged, or
a loop $\a \lra \b \lra \a$ is replaced by 
$\a$.  In the latter case, the element of the algebra
is multiplied by $t_i$, by relation \Athree\  in the definition of the algebra $A$.
This is also the only change which affects the numbers $\theta_j(\phi_t)$; it decreases
$\theta_i$ by one and leaves the other $\theta_j$ alone.

Thus $p(\phi) = p(\phi_1) \cdot \prod_{i\in I} t_i^{\theta_i}$, and
$\phi_1$ represents a taut path, since a line segment 
cannot cross any hyperplane more than once.
\end{proof}

\begin{corollary}\label{taut}
Let
\[
	\a = \a_1 \lra \a_2 \lra \dots \lra \a_d = \b
\]
and
\[
	\a = \b_1 \lra \b_2 \lra \dots \lra \b_d = \b
\]
be two taut paths between fixed elements $\a,\b \in \cP$.
Then
\[
	p(\a_1, \dots, \a_k) = p(\b_1, \dots, \b_d).
\]
\end{corollary}
\begin{proof}
The proof of Proposition \ref{standard form}
shows that we can write $p(\a_1, \dots, \a_k) = p(\phi_1)$
and $p(\b_1, \dots, \b_d) = p(\phi_2)$, where
$\phi_1$ and $\phi_2$ are both straight-line paths from a 
 point of $\Delta_{\a_1}$ to a point of $\Delta_{\a_k}$.
These endpoints can be chosen arbitrarily from a
dense open subset of $\Delta_{\a_1}\times \Delta_{\a_k}$,
so we can take $\phi_1 = \phi_2$.
\end{proof}

\begin{corollary} \label{visits every region}
Consider an element $$a = p\cdot \prod_{i\in I}t_i^{d_i} \in e_\a Ae_\b,$$
where $p$ is represented by a taut path from $\a$ to $\b$ in $Q$.
Suppose that $\gamma \in \cF$ satisfies $\gamma(i) = \a(i)$
whenever $\a(i) = \b(i)$ and $d_i = 0$.
Then $a$ can be written as an $\R$-linear combination of elements represented by
paths in $Q$ all of which pass through $\ga$.

In particular, if $\gamma \in \cF\ssm \cP$, then $a = 0$.
\end{corollary}

\begin{proof}
Applying Proposition \ref{standard form} and Corollary \ref{taut}
we see that the composition of taut 
paths from $\a$ to $\ga$ and from $\ga$ to $\b$ is equal in $A$
to $p \cdot \prod_{i\in I} t_i^{d'_i}$, where $d'_i$ is $1$
if $\a(i) = \b(i) \ne \ga(i)$ and is $0$ otherwise.  
The result follows since $d_i \ge d'_i$ for all $i$ 
and for any $i$ with $d_i > d'_i$ we
can express $e_\b t_i$ as a combination of paths
$p(\b,\b',\b)$, as in Lemma \ref{smaller quiver}.
\end{proof}

\subsection{Quadratic duality}\label{qd}
We conclude this section by establishing the first part of Theorem (A), 
namely that the algebras $A(\cV)$ and $A(\cV^\gd)$
are quadratic duals of each other.  First we review the definition 
of quadratic duality.  See \cite{PolPos} for more about quadratic algebras.

Let $R = \R\{e_{\a}\mid \a\in \cI\}$ be a ring spanned by finitely many
pairwise orthogonal idempotents, and let $M$ be an $R$-bimodule.  Let
$T_R(M)$ be the tensor algebra of $M$ over $R$, and let $W$ be a
sub-bimodule of $M\otimes_R M = T_R(M)_2$.  For shorthand, we will write $M_{\a\b} = e_\a M e_\b$
and $W_{\a\b} = e_\a W e_\b$.
To this data is associated a quadratic algebra
$$E = T_R(M)\,\Big/\, T_R(M)\cdot W\cdot T_R(M).$$

The {\bf quadratic dual} $E^\qd$ of $E$ is defined as the quotient
$$E^\qd = T_R(M^*)\,\Big/\, T_R(M^*)\cdot W^\perp\cdot T_R(M^*),$$
where $M^*$ is the vector space dual of $M$, and
$$W^\perp \subset M^*\otimes_R M^* \cong \big(M\otimes_R M\big)^*$$
is the space of elements that vanish on $W$.  
Note that dualizing $M$ interchanges the left and right $R$-actions,
so that $(M_{\a\b})^* = (M^*)_{\b\a}$.
It is clear from this definition that 
there is a natural isomorphism $E^{\qd\qd}\cong E$.

As we saw in Lemma \ref{smaller quiver}, the algebra $A(\cV)$
is quadratic, where we take $\cI = \cP$  and 
$$M = \R\{\, p(\a,\b)\mid \a,\b\in\cP\,\,\text{such that}\,\,\a\lra\b\}.$$
The relations of type \Atwo\ from Definition \ref{defining A} lie in $W_{\a\gamma}$,
while relations of type \Athree\ lie in $W_{\a\a}$.  Since $\cP^\gd = \cP$,
the algebra $A(\cV^\gd)$ for the Gale dual polarized arrangement 
has the same base ring $R$.
The degree one generating sets are also canonically isomorphic, since
the adjacency relations on $\cP$ and $\cP^\gd$ are the same.  However,
in order to keep track of which algebra is which, we will denote the 
generators of $A(\cV^\gd)$ by $p^\gd(\a,\b)$ and their span by $M^\gd$.

We want to show that $A(\cV)$ and $A(\cV^\gd)$ are quadratic dual rings, so
we must define a perfect pairing $$M\otimes M^\gd\to\R$$ to identify $M^\gd$ with $M^*$.
An obvious way to do this would be to make $\{p^\gd(\a,\b)\}$ the dual
basis to $\{p(\a,\b)\}$,
but this does not quite give us what we need.  Instead, we need to twist this pairing
by a sign. Choose a subset $X$ of edges in the underlying undirected graph of $Q$
with the property that 
for any square $\a \lra \b \lra \gamma \lra \delta \lra \a$
of distinct elements, an odd number of the 
edges of the square are in $X$.
The existence of such an $X$ follows from the fact that our graph is 
a subgraph of the edge graph of an $n$-cube.
Define a pairing $\langle\, ,\rangle$ by putting
\begin{equation*}
	\langle\, p(\a, \b),\, p^\gd(\b,\a)\, \rangle =
        \begin{cases}
          -1 & \text{if}\, \{\a, \b\} \in X\\
          1 & \text{if}\, \{\a,\b\}\notin X
        \end{cases}
\end{equation*}
and $$\langle\, p(\a, \b),\, p(\delta,\gamma)\, \rangle = 0$$
unless $\a=\gamma$ and $\b=\delta$.


\begin{theorem}\label{Quadratic dual}
The above pairing induces an isomorphism $A(\cV^\gd)\cong A(\cV)^\qd$.
\end{theorem}

\begin{proof} Let $W \subset (M\otimes_R M)_2$ and 
$W^\gd \subset (M^\gd \otimes_R M^\gd)_2$ be the spaces
of relations of $A(\cV)$ and $A(\cV^\gd)$, respectively.
We analyze each piece $W_{\a\ga}$ and $W^\gd_{\ga\a}$ for every pair
$\a, \gamma \in \cP$ which admit paths of 
length two connecting them.  First consider the 
case $\a \ne \gamma$; they must differ in exactly two entries, 
so there are exactly two elements $\b_1, \b_2$ in $\{\pm 1\}^I$ 
which are adjacent to both $\a$ and $\ga$.
Since we are assuming that there is 
a path from $\a$ to $\gamma$ in $\cP$, at least
one of the $\b_j$ must be in $\cP$.

If both $\b_1$ and $\b_2$ are in $\cP$,
then $e_\a M\otimes_R M e_\gamma$ 
is a two-dimensional vector space with basis
$\{p(\a,\b_1) \otimes p(\b_1,\gamma), p(\a,\b_2)\otimes p(\b_2,\gamma)\}$,
while $e_\gamma M\otimes_R M e_\a$ 
has a basis $\{p^\vee(\gamma,\b_1)\otimes p^\gd(\b_1,\a), p^\vee(\gamma,\b_2)\otimes p^\gd(\b_2,\a)\}$.
Then the relation \Atwo\ gives
\[W_{\a\ga} = \R\{p(\a,\b_1) \otimes p(\b_1,\gamma) - p(\a,\b_2)\otimes p(\b_2,\gamma)\}\; \text{and}\]
\[W^\gd_{\ga\a} = \R\{p^\vee(\gamma,\b_1)\otimes p^\gd(\b_1,\a) - p^\vee(\gamma,\b_2)\otimes p^\gd(\b_2,\a)\},\]
and so $W^\gd_{\ga\a} = (W_{\a\ga})^\perp$ (this is where we use the signs in our pairing).

If $\b_1 \in \cP$ and 
$\b_2 \notin \cP$, then either $\b_2 \in \cF \ssm \cP$ and
$\b_2 \notin \cF^\gd$ or $\b_2 \in \cF^\gd \ssm \cP^\gd$ and
$b_2 \notin \cF$.  In the first case we get 
$W_{\a\gamma}=e_{\a}M\otimes_R Me_{\ga}$ 
since $p(\a,\b_1,\ga) = p(\a,\b_2,\ga) = 0$ in 
$A(\cV)$ by relations \Aone\ and \Atwo. On the other
hand $\b_2 \notin \cF^\gd$ means that 
the relation \Atwo\ doesn't appear on the dual side, so 
$W^\gd_{\ga\a} = 0 = (W_{\a\ga})^\perp$.  
The argument in the second case is the same, reversing the role of
$\cV$ and $\cV^\gd$. 

We have dealt with the case $\a\neq\gamma$, so suppose now that
$\a = \gamma$.  The vector space $e_\a M \otimes_R Me_\a$ 
has a basis consisting of the elements $p(\a,\a^i) \otimes p(\a^i,\a)$
where $i$ lies in 
the set $$J_\a := \{i \in I \mid \a^i \in \cP\}.$$
(Recall that $\a^i$ is the element of 
$\{\pm 1\}^I$ which differs from $\a$ in 
precisely the $i^\text{th}$ entry.)
Note that $J_\a$ is a subset of the set $I_\a$ defined in the 
proof of Lemma 
\ref{smaller quiver}.
Let $I^\gd_\a$ and $J^\gd_\a$ be the corresponding sets for $\cV^\gd$,
and note that $J^\gd_\a = J_\a$.

Using this basis we can identify $W_{\a\a}$ with a subspace
of $\R^{J_\a}$, which we compute as follows. 
Given a covector $w \in (\R^I)^* \cong \R^I$, its image
in $V^*$ is $\sum_{i \in I} w_i t_i$, so the relations 
among the $t_i$ in 
$\Sym(V^*)$ are given by $V^\bot \subset (\R^I)^*\cong \R^I$.
If $\sum_{i\in I}w_it_i = 0$ and $w_i = 0$ for $i \notin I_\a$,
then multiplying by $e_\a$ and using relation \Aone\ and \Athree\ gives
\[\sum_{i\in J_\a} w_ip(\a,\a^i,\a) = \sum_{i \in I_\a}w_i p(\a,\a^i,\a) = 0.\]
Thus $W_{\a\a}$ is the projection of $V^\perp \cap \R^{I_\a}$ onto $\R^{J_\a}$.
Alternatively, we can first project $V^\perp$ onto 
$\R^{I_\a^c\cup J_\a}$ and then intersect with $\R^{J_\a}$.

On the dual side, the space $W^\gd_{\a\a}$ 
of relations among loops at $\a$ is identified with a vector subspace
of $\R^{J^\gd_\a} = \R^{J_\a}$, namely the projection of 
$V\cap \R^{I^\gd_\a}$ onto $\R^{J^\gd_\a}$.
This is the orthogonal complement of $W_{\a\a}$, since 
Lemma \ref{duality lemma} implies that $I_\a \ssm J_\a = I \ssm I^\gd_\a$.
Note that the signs we added to the pairing do not affect this,
since we always have
\[\langle p(\a,\b), p^\gd(\b,\a)\ra\cdot\langle p(\b,\a),p^\gd(\a,\b)\ra = 1.\qedhere\]
\end{proof}

\section{The algebra $B$}\label{sec:convolution}
\subsection{Combinatorial definition}\label{combdef}
In this section we define our second algebra $B(\cV)$ associated to the 
polarized arrangement $\cV = (V,\eta,\xi)$.
Consider the polynomial ring $\R[u_i]_{i \in I} = \Sym\R^I$.  
We give it a grading by putting $\deg u_i = 2$ for all $i$. 
For any subset $S \subset I$, let 
$u_S = \prod_{i \in S} u_i$.

\begin{definition}
Given a subset $\Delta \subset V_\eta$,
define graded rings 
\begin{equation*}
 \tilde{R}_{\Delta} := \Sym\R^I\big/
\left\langle u_S\mid
S\subset I\,\,\text{such that}\,\, 
H_S \cap \Delta = \emptyset\right\rangle
\end{equation*}
and $$R_{\Delta} := \tilde{R}_{\Delta} \otimes_{\Sym V} \R,$$
where the map from $\Sym V$ to $\Sym \R^I$
is induced by the inclusion $V \hookrightarrow \R^I$,
and the map $\Sym(V) \to \R$ is the graded map which 
kills $V$.
We use the conventions that $u_{\emptyset} = 1$ and 
$H_\emptyset = V_\eta$, 
which means that
$R_\Delta = 0$ if and only if $\Delta = \emptyset$.
Notice that if $\Delta_1 \subset \Delta_2$, then we have  natural 
quotient maps $\tilde{R}_{\Delta_2} \to \tilde{R}_{\Delta_1}$ and 
$R_{\Delta_2} \to R_{\Delta_1}$.

There are only certain subsets that will interest us,
and for each of these subsets we introduce simplified notation for the corresponding ring.
We put 
\begin{equation}\label{rch}
\tilde{R}_\cH = \tilde{R}_{V_\eta},\end{equation}
and for any $\a,\b,\ga,\de \in \cF$,
\begin{equation}\label{rabgd}
\tilde{R}_{\a} = \tilde{R}_{\Delta_\a},\,\,\,
\tilde{R}_{\a\b} = \tilde{R}_{\Delta_\a \cap \Delta_\b},\,\,\,
\tilde{R}_{\a\b\ga} = \tilde{R}_{\Delta_\a \cap \Delta_\b\cap \Delta_\ga},\,\,\,
\text{and}\,\,\, \tilde{R}_{\a\b\ga\de} = \tilde{R}_{\Delta_\a \cap \Delta_\b\cap \Delta_\ga\cap \Delta_\de}.\end{equation}
Finally, we let $R_\cH$, $R_\a$. $R_{\a\b}$, $R_{\a\b\ga}$, and $R_{\a\b\ga\de}$ 
denote the tensor products of these rings with $\R$ over $\Sym(V)$.
\end{definition}

\begin{lemma}\label{formality}
The ring $\tilde{R}_\cH$ is a free $\Sym(V)$-module of total rank
$|\Bas|$, the number of bases of the matroid of $V$.
For any intersection $\Delta$ of chambers of $\cH$, the
ring $\tilde{R}_{\Delta}$ is a free $\Sym(V)$-module of total rank
$|\Bas_\Delta|$, where $\Bas_\Delta = \left|\{b \in \Bas \mid H_b \subset \Delta\}\right|$.
\end{lemma}

\begin{proof}
Both of these rings are the face rings of shellable 
simplicial complexes, namely the matroid complex
of $\cV$ and the dual of the face lattice of 
$\Delta$, respectively.  This implies that they are
free over the symmetric algebra of \emph{some} subspace
of $\R^I$, with bases parametrized by the top-dimensional
simplices, which are in bijection with $\Bas$ and 
$\Bas_\Delta$, respectively.
The fact that the rings are free over $\Sym(V)$ specifically
follows from
the fact that $V$ projects isomorphically onto $\R^b$
for any basis $b$.
\end{proof}

As we will see in the next section, when the arrangement $\cV$ is
rational, the rings of Equations \eqref{rch} and \eqref{rabgd} can be interpreted as 
equivariant
cohomology rings of algebraic varieties with torus actions, while their tensor products with $\R$
are the corresponding ordinary cohomology rings.
(In particular, we will give a topological interpretation of Lemma \ref{formality}
in Remark \ref{fixed points}.)
However, our main theorems will all be proved in a purely algebraic setting.

For $\a, \b\in \cP$, let
\[d_{\a\b} = \big | \{i \in I \mid \a(i) \neq \b(i)\}\big |.\] 
If the intersection $\Delta_\a \cap \Delta_\b$ is nonempty, then $d_{\a\b}$
is equal to its codimension inside of $V_\eta$.
For $\a,\b,\ga\in\cP$, let
\[S(\a\b\ga) = \{i \in I \mid \a(i) = \ga(i) \ne \b(i) \}.\]
If $\a,\b$, and $\ga$ are all feasible, then $S(\a\b\ga)$ is the set of hyperplanes
that are crossed twice by the composition of a pair of taut paths from
$\a$ to $\b$ and $\b$ to $\ga$.

\begin{definition}  Given a polarized arrangement $\cV$, let
\[B = B(\cV) := \bigoplus_{(\a,\b) \in \cP \times \cP} R_{\a\b}[-d_{\a\b}].\]
We define a product operation
\[\conv\colon B \otimes B \to B\]
via the composition
\[\xymatrix{
R_{\a\b} \otimes R_{\b\ga} \ar[r] & R_{\a\b\ga}\otimes R_{\a\b\ga} \ar[r] & R_{\a\b\ga} \ar[rr]^{\cdot u^{}_{S(\a\b\ga)}} && R_{\a\ga}
},\]
where the first map is the tensor product of the natural quotient maps, the second is multiplication in
$R_{\a\b\ga}$, and
the third is induced by multiplication by the monomial $u^{}_{S(\a\b\ga)}$.
To see that the third map is well-defined, it is enough to observe that 
\[\Delta_\a \cap \Delta_\b \cap \Delta_\de = \Delta_\a \cap \Delta_\b \cap H_{S(\a\b\de)}.\]
(All tensor products above are taken over $\R$, and the product is identically zero on
$R_{\a\b}\otimes R_{\de\ga}$ if $\de\neq \b$.)
\end{definition}

\begin{proposition}
 The operation $\conv$ makes $B$ into a graded ring.
\end{proposition}
\begin{proof}
The map $x \otimes y \otimes z \mapsto (x \conv y) \conv z$ from
$R_{\a\b} \otimes R_{\b\ga} \otimes R_{\ga\de}$ to $R_{\a\de}$ can also be computed by 
multiplying the images of $x$, $y$, and $z$ in $R_{\a\b\ga\de}$, and 
then mapping into $R_{\a\de}$ by multiplication by $u^{}_{S(\a\b\ga)}u^{}_{S(\a\ga\de)}$.
As a result, associativity of $\conv$ follows from the identity
$u^{}_{S(\a\b\ga)}u^{}_{S(\a\ga\de)} = u^{}_{S(\b\ga\de)}u^{}_{S(\a\b\de)}$.  This 
easy to verify by hand: the power to which the variable $u_i$ appears on either side
is $\theta_i(\a,\b,\ga,\de)$ (recall Definition \ref{recursion statistic}).
The fact that the product is compatible with the grading follows from the identity
\[d_{\a\b} + d_{\b\ga} - d_{\a\ga} = 2\, \big |S(\a\b\ga)\big |\]
for all $\a, \b, \ga \in \cP$. 
\end{proof}

\begin{remark}\label{Bprime}
The graded vector space $\tilde{B} := \bigoplus_{(\a,\b) \in \cP \times \cP} \tilde{R}_{\a\b}[-d_{\a\b}]$
can be made into a graded ring in exactly the same way.
There is a natural ring homomorphism
$$\zeta:\Sym\R^I\hookrightarrow\bigoplus_{\a\in\cP}\Sym\R^I\twoheadrightarrow
\bigoplus_{\a\in\cP}\tilde{R}_{\a\a}\hookrightarrow\tilde{B}$$
making $\tilde{B}$ into an algebra over $\Sym\R^I$, and we have
$$B\cong \tilde{B}\otimes_{\Sym\R^I}\Sym(\R^I/V)\cong \tilde{B}\otimes_{\Sym V}\R.$$
(compare to Remark \ref{deligne}).
In \cite{BLP+} we construct a canonical deformation of any quadratic algebra,
and the algebras $\tilde{A}$ and $\tilde{B}$ are the canonical deformations of $A$ and $B$,
respectively.
\end{remark}



\subsection{Toric varieties}\label{toric}
We next explain how the ring $B(\cV)$ arises from the geometry
of toric and hypertoric varieties in the case where $\cV$ is rational.  
We begin by using the data in
$\cV$ to define a collection of toric varieties, which 
are lagrangian subvarieties 
of an algebraic symplectic orbifold $\M_\cH$, the 
hypertoric variety determined by the arrangement $\cH$.

Let $\cV = (V,\eta,\xi)$ be a rational polarized arrangement.
The vector space $V$ inherits an integer lattice $V_\Z := V\cap \Z^I$,
and the dual vector space $V^*$ inherits a dual lattice which is a quotient
of $(\Z^I)^*\subset(\R^I)^*$.  Consider the
compact tori $$T^I := (\R^I)^*/(\Z^I)^*\twoheadrightarrow
V^*/V^*_\Z =: T.$$
For every feasible chamber $\a\in\cF$, the polyhedron $\Delta_\a$ determines
a toric variety with an action of the complexification $T_\C$.
The construction that we give below is originally due to Cox \cite{Cox}.

Let $T^I_\C$ be the complexification of $T^I$,
and let $\C^I_\a$ be the representation of $T^I_\C$ in which the $i^\text{th}$
coordinate of $T^I_\C$ acts on the $i^\text{th}$ coordinate
of $\C^I_\a$ with weight $\a(i)\in\{\pm 1\}$.
Let $$G = \ker(T_\C^I\twoheadrightarrow T_\C),$$ 
and  let $$Z_\a\,\, =\,\, \C_\a^I \,\,\ssm \bigcup_{\substack{S\subset I\,\, 
\text{such that}\\ H_S \cap \Delta_\a = \emptyset}}
\{z\in \C^I \mid z_i = 0 \,\text{ for all } i\in S\},$$
which is acted upon by $T^I_\C$ and therefore by $G$.
The toric variety $X_\a$ associated to $\Delta_\a$ is defined to be the quotient
$$X_\a = Z_\a/G,$$ and it inherits an action of the quotient torus $T_\C = T^I_\C/G$.
The action of the compact subgroup $T\subset T_\C$ is hamiltonian with respect
to a natural symplectic structure on $X$, and $\Delta_\a$ is the moment polyhedron.
If $\Delta_\a$ is compact\footnote{We use the word "compact" rather than the more standard word "bounded"
to avoid confusion with the fact that $\a$ is called bounded if $\xi$
is bounded above on $\Delta_\a$.},
then $X_\a$ is projective.  More generally, $X_\a$ is projective
over the affine toric variety whose coordinate ring is the semi-group ring of
the semi-group $\Sigma_\a\cap V_\Z$.
Since the polyhedron $\Delta_\a$ is simple by our assumption on $\eta$,
the toric variety $X_\a$ has at worst finite quotient singularities. 
 
Now let $$\wt{\eX}=\coprod_{\a \in \cP} X_\a$$ 
be the disjoint union 
of the toric varieties associated to the bounded feasible sign vectors,
that is, to the chambers of $\cH$ on which $\xi$ is bounded above.
We will define a quotient space $\eX$ of $\wt{\eX}$ which, informally,
is obtained by gluing the components of $\tilde{\eX}$ together along
the toric subvarieties corresponding to the faces at which the 
corresponding polyhedra intersect.

More precisely, let $\C^I$ be the standard coordinate representation
of $T^I_\C$, and let $$\H^I = \C^I \times (\C^I)^*$$ be the 
product of $\C^I$ with its dual.  For every $\a\in\{\pm 1\}^I$,
$\C^I_\a$ can be found in a unique way as a subrepresentation of $\H^I$.
Then $$\eX := \left(\,\bigcup_{\a\in\cP} Z_\a\right)\Big{/}\, G,$$ where the union
is taken inside of $\H^I$.  Then $T_\C$ acts on $\eX$,
and we have a $T_\C$-equivariant projection
$$\pi:\wt{\eX}\to\eX.$$
The restriction of this map to each $X_\a$ is obviously an embedding.
As a result, any face of a polyhedron $\Delta_\a$ corresponds to a
$T$-invariant subvariety which is itself a toric variety for a 
subtorus of $T_\C$.

The singular variety $\eX$ sits naturally as a closed subvariety 
of the hypertoric variety $\M_\cH$, which is defined
as an algebraic symplectic quotient of $\H^I$ by $G$ (or, equivalently,
as a hyperk\"ahler quotient of $\H^I$ by the compact form of $G$).
See \cite{Pr07} for more details.
The subvariety $\eX$ is lagrangian, and is
closely related to two other lagrangian subvarieties of $\M_\cH$
which have appeared before in the literature.  
The projective components of $\eX$ (the components
whose corresponding polyhedra are compact) form a complete list of 
irreducible projective lagrangian subvarieties of $\M_\cH$, and their union
is called the {\bf core} of $\M_\cH$.  On the other hand, 
we can consider the larger subvariety 
$\ecore = (\bigcup_{\alpha\in \cF} Z_\a) / G$, where the union is taken
over all feasible chambers, not just the bounded ones.  This larger union
was called the {\bf extended core} in \cite{HP04};
it can also be described as the zero level of the moment map for the hamiltonian
$T_\C$-action on $\M_\cH$.  

Our variety $\eX$, which sits in between the core 
and the extended core, will be referred to as the {\bf relative core} of $\M_\cH$
with respect to the $\C^\times$-action defined by $\xi$.
It may be characterized as the set of points 
$x\in \M_\cH$ such that $\lim_{\la\to\infty} \la\cdot x$ exists.

\begin{example}\label{an}
Suppose that $\cH$ consists of $n$ points in a line.  Then $\M_\cH$
is isomorphic to the minimal resolution of $\C^2/\Z_n$, and its core is equal
to the exceptional fiber of this resolution, which is a chain of $n-1$ projective
lines.  The extended core is larger; it includes two affine lines attached to the projective
lines at either end of the chain.  The relative core lies half-way in between, containing exactly
one of the two affine lines.  This reflects the fact that $\xi$ is bounded above on exactly
one of the two unbounded chambers of $\cH$.

For this example, the category of ungraded right $B$-modules is 
equivalent to the category of perverse sheaves on $\mathbb{P}^{n-1}$
which are constructible for the Schubert stratification.
This in turn is equivalent to a regular block of
parabolic category $\cO$ for $\mathfrak{g} = \mathfrak{sl_n}$
and the parabolic $\mathfrak{p}$ whose associated Weyl group is
$S_1 \times S_{n-1}$ \cite[1.1]{Str06a}.
\end{example}

The core, relative core, and extended core of $\M_\cH$ are all $T$-equivariant
deformation retracts of $\M_\cH$, which allows us to give a combinatorial description
of their ordinary and equivariant cohomology rings \cite{Ko,HS,Pr07}.

\begin{theorem}\label{hypertoric cohom}
There are natural isomorphisms
$$H^*_T(\eX) \,\,\cong\,\, H^*_T(\M_\cH)\,\,\cong\,\,
\tilde{R}_\cH$$
and
$$H^*(\eX) \,\,\cong\,\, H^*(\M_\cH)\,\,\cong\,\,
R_\cH.$$
\end{theorem}

We have a similar description of the ordinary and equivariant cohomology
of the toric components $X_\a$ and their intersections.  
For $\a,\b,\ga\in\cP$,
let $$X_{\a\b} = X_\a \cap X_\b\,\,\,\text{and}\,\,\, 
X_{\a\b\gamma} = X_\a\cap X_\b\cap X_\gamma,$$
where the intersections are taken inside of $\eX$.  

\begin{theorem}\label{toric cohom}
 There are natural isomorphisms
\[H^*_T(X_{\a\b}) \cong \tilde{R}_{\a\b},\,\,\, H^*_T(X_{\a\b\ga}) \cong \tilde{R}_{\a\b\ga},\]
\[H^*(X_{\a\b}) \cong R_{\a\b},\,\,\, H^*(X_{\a\b\ga}) \cong R_{\a\b\ga}.\]
Under these isomorphisms and the isomorphisms of Theorem \ref{hypertoric cohom},
the pullbacks along the inclusions $X_{\a\b\ga} \to X_{\a\b}$ and $X_{\a\b} \to \eX$ 
are the natural maps induced by the identity map on $\Sym\R^I$.
\end{theorem}
\begin{proof}
The existence of these isomorphisms is well-known, but in order to pin down
the maps between them, we carefully explain exactly how our isomorphisms
arise.
Since the action of $G$ on $Z_\a \cap Z_\b$ is locally free, we have
\[H^*_T(X_{\a\b}) = H^*_{T^I_\C/G}((Z_\a \cap Z_\b)/G) \cong H^*_{T^I_\C}(Z_\a \cap Z_\b).\]
A result of Buchstaber and Panov \cite[6.35 \& 8.9]{BP} computes the
equivariant cohomology of the complement of any union of equivariant 
subspaces of a vector space with a torus action in which the generalized
eigenspaces are all one-dimensional.
Applied to $Z_\a \cap Z_\b$, this gives the ring $\tilde{R}_{\a\b}$.  More precisely, the restriction map
\[\Sym\R^I = H^*_T(\C^I_\a \cap \C^I_\b) \to H^*_{T^I_\C}(Z_\a \cap Z_\b)\]
is surjective, with kernel equal to the defining ideal of $\tilde{R}_{\a\b}$.
(Note that our torus $T^I_\C$ is of larger dimension than
the affine space $\C^I_\a \cap \C^I_\b$.  The ``extra'' coordinates in $T^I_\C$
act trivially, and correspond to the variables $u_i$ with 
$\Delta_\a \cap \Delta_\b \subset H_i$, which do not appear in the relations 
of $\tilde{R}_{\a\b}$).  

The identification of $H^*_T(X_{\a\b\ga})$ with $\tilde{R}_{\a\b\ga}$ follows similarly,
and the computation of the pullback by $X_{\a\b\ga} \to X_{\a\b}$ follows
because the restriction
\[\Sym\R^I\cong H^*_{T^I_\C}(\C^I_\a \cap \C^I_\b) \to H^*_{T^I_\C}(\C^I_\a \cap \C^I_\b \cap \C^I_\ga)\cong\Sym\R^I\]
is the identity map. 

The restriction $H^*_T(\eX) \to H^*_T(X_{\a\b})$ is computed by a similar 
argument: the proof of \cite[3.2.2]{Pr07} uses an isomorphism
 $H^*_T(\eX) \cong H^*_{T^I_\C}(U)$, where $U \subset \C^I$ is an 
open set, and the restriction 
$\Sym\R^I \cong H^*_{T^I_\C}(\C^I) \to H^*_{T^I_\C}(U)$ is surjective.
\end{proof}

\begin{remark}\label{fixed points}
With these descriptions of $\tilde{R}_\cH$ and
$\tilde{R}_{\a\b}$ as equivariant cohomology rings, 
Lemma \ref{formality} is a consequence of the 
equivariant formality of the varieties
$\M_\cH$ and $X_{\a\b}$, and the fact that we have
bijections $\M_\cH^T \lra \Bas$ and
$X_{\a\b}^T \lra \Bas_{\Delta_{\a\b}}$.
\end{remark}

\subsection{A convolution interpretation of \boldmath{$B$}}\label{section:topological B}
For $\cV$ rational, Theorem \ref{toric cohom} gives isomorphisms
$$B \cong \bigoplus_{(\alpha,\beta) \in \cP\times \cP} H^*(X_{\a\b})[-d_{\a\b}]
\,\,\cong\,\, H^*(\wt{\eX} \times_\pi \wt{\eX})$$ and
$$\tilde{B} \cong \bigoplus_{(\alpha,\beta) \in \cP\times \cP} H^*_T(X_{\a\b})[-d_{\a\b}]
\,\,\cong\,\, H_T^*(\wt{\eX} \times_\pi \wt{\eX}),$$
where the (ungraded) isomorphisms on the right follow from the fact that 
$$\wt{\eX} \times_\pi \wt{\eX} = \coprod_{(\alpha,\beta) \in \cP\times \cP} X_{\a\b}.$$
We next show how to use these isomorphisms to interpret the product $\conv$
geometrically.

The components $X_{\a\b}$ of $\wt{\eX} \times_\pi \wt{\eX}$
all have orientations coming from their complex structure, but we will
twist these orientations by a combinatorial sign.
For each $\a,\b\in\cP$, we give 
$X_{\a\b}$ $(-1)^n$ times the 
complex orientation, where $n$ is the number of $i\in I$
with $\a(i) = \b(i) = -1$.
Geometrically, these are the indices for which 
the polytope $\Delta_\a \cap \Delta_\b$ lies on the 
negative side of the hyperplane $H_i$.  We use a similar rule
to orient the components $X_{\a\b\ga}$ of $\wt{\eX} \times_\pi \wt{\eX} \times_\pi \wt{\eX}$.

Let
\[
	p_{12},p_{13},p_{23}: \wt{\eX} \times_\pi \wt{\eX} \times_\pi \wt{\eX} \longrightarrow
	\wt{\eX} \times_\pi \wt{\eX}
\]
denote the natural projections.  Note that these maps are proper; they are finite disjoint unions
of closed immersions of toric subvarieties.

\begin{proposition}\label{geom b}
The product operations on $B(\cV)$ and $\tilde{B}(\cV)$ are given by 
$$a \conv b = ({p_{13}})_* \big( p_{12}^*(a) \cup p_{23}^*(b) \big),$$
where ${p_{13}}_*$ is the Gysin pushforward relative to the given twisted
orientations.
\end{proposition}

\begin{proof}
For an approach to defining the Gysin pushforward in equivariant 
cohomology, see Mihalcea \cite{Mih}.  It is only defined
there for  maps between projective varieties, but it is
easily extended to general proper maps to smooth varieties,
using the Poincar\'e duality isomorphism between cohomology and Borel-Moore
homology.  

For any $\a, \b, \ga\in \cP$, consider the diagram
\[\xymatrix{
H_{T^I_\C}^*(\C^I_\a \cap \C^I_\b \cap \C^I_\ga) \ar[r]\ar[d] & H^*_{T^I_\C}(Z_\a \cap Z_\b \cap Z_\ga) \ar[r]^{\,\,\,\,\,\,\,\,\cong}\ar[d] & H^*_T(X_{\a\b\ga}) \ar[r]^{\cong}\ar[d] &
 \tilde{R}_{\a\b\ga}\ar^{\cdot u^{}_{S(\a\b\ga)}}[d] \\
H_{T^I_\C}^*(\C^I_\a \cap \C^I_\ga) \ar[r] & H^*_{T^I_\C}(Z_\a \cap Z_\ga) \ar[r]^{\cong} & H^*_T(X_{\a\ga}) \ar[r]^{\cong} & \tilde{R}_{\a\ga}
 }\]
where the horizontal maps on the left are 
restrictions, the middle maps are the 
natural isomorphisms induced by taking the quotient by $G$, 
the right-hand maps are the isomorphisms of Theorem \ref{toric cohom}, and the first three vertical
maps are Gysin pushforwards.  (We give the intersections of the $\C^I_\a$ and the $Z_\a$
orientations compatible with those on the corresponding toric varieties.)  
Our proposition is the statement that the square on the right commutes.
Since the left and middle squares commute, it will be enough to show that
the left Gysin map is given by multiplication by $u^{}_{S(\a\b\ga)}$.  Indeed, this
map is multiplication by the equivariant Euler class of the 
normal bundle of $\C^I_\a \cap \C^I_\ga$ in $\C^I_\a \cap \C^I_\b \cap \C^I_\ga$.
If these spaces were given the complex orientation, this would be 
the product of the $T^I_\C$-weights of the quotient representation, 
which is $\prod_{i \in S(\a\b\ga)} \a(i)u_i$.  But each eigenspace with $\a(i) = -1$ has been
given the anti-complex orientation, so the signs disappear and we are left with
multiplication by $u^{}_{S(\a\b\ga)}$, as required.
\end{proof}

\begin{remark}
 This convolution product is similar to one defined by Ginzburg \cite{CG97} on the
Borel-Moore homology of a fiber product $Y \times_\pi Y$ for a map
$\pi:Y \to X$ where $Y$ is smooth.  The Ginzburg ring is different from ours, however: 
it uses the intersection product in $Y \times Y$, whereas our cup product takes place in
the fiber product $Y \times_\pi Y$.  Ginzburg's convolution product is 
graded and associative without degree shifts or twisted orientations, while our product
requires these modifications.
\end{remark}

\begin{remark}
The fact that each component $X_\a$ of $\eX$
can be thought of as an irreducible lagrangian subvariety of the hypertoric variety
$\M_\cH$ allows us to interpret the cohomology groups of their intersections
as Floer cohomology groups.  From this perspective, 
$B$ can be understood as an Ext-algebra in the Fukaya category of $\M_\cH$.
This description should be related to the description in
Remark~\ref{D-mod} by taking homomorphisms to the canonical
coisotropic brane, as described by Kapustin and Witten \cite{KW07}.
\end{remark}

\begin{remark}
  Stroppel and the the fourth author considered an analogous convolution 
  algebra using the components of a Springer fiber
  for a nilpotent matrix with two Jordan blocks (along with some associated 
  non-projective varieties of the same dimension) in place of 
  the toric varieties $X_{\a}$ \cite{SW08}.  
  They show that right modules over this algebra are equivalent to a
  block of parabolic category $\cO$ for a maximal parabolic of
  $\mathfrak{sl}_n$.  Thus the category of right $B$-modules
can be thought of as an
  analogue of Bernstein-Gelfand-Gelfand's category $\cO$ in a combinatorial
  (rather than Lie-theoretic) context.  In Sections \ref{sec:repr-categ} and
  \ref{derived equivalences}, we will show that this category shares 
  many important properties with category $\cO$.
\end{remark}

\subsection{\boldmath{$A$} and \boldmath{$B$}}
We now state and prove the first main theorem of this section, which, along with Theorem \ref{Quadratic dual},
comprises Theorem (A) from the Introduction.

\begin{theorem}\label{convolution equals quiver}
There is a natural isomorphism $A(\cV^\vee)\cong B(\cV)$ of graded rings.
\end{theorem}

\begin{proof}
We define a map $\phi\colon A(\cV^\gd) \to B(\cV)$ by
\begin{itemize}
 \item sending the idempotent $e_\alpha$ to the unit element
$1_{\a\a} \in R_{\a\a}$ for all
$\a \in \cP$,
\item sending $p(\a,\b)$ to the unit
$1_{\a\b} \in R_{\a\b}$ for all
$\a, \b\in \cP$ with $\a \lra \b$, and
\item sending $t_{i} \in (V^\bot)^*\cong \R^I/V$ to 
$\zeta(u_i)$.

\end{itemize}
To show that this is a homomorphism, 
we need to check that these elements satisfy the 
relations \Atwo\  and \Athree\  from Definition \ref{defining A}
(for $\cV^\gd$).
In order to check that relation \Atwo\  holds, 
suppose that $\a,\b,\gamma,\delta \in \cF^\gd$ are distinct and 
satisfy 
$\a \lra \b \lra \gamma\lra \delta \lra \a$, and
$\a, \gamma \in \cP^\gd = \cP$.  Since they are distinct, 
we must have
$\a \lrao i \b \lrao j \gamma\lrao i \delta \lrao j \a$
for some $i, j \in I$ with $i \ne j$.  It follows that
$S(\a\b\ga) = S(\a\de\ga) = \emptyset$.

There are two possibilities:
first, if $\b$ and $\delta$ also lie in $\cP^\gd$, then
$$1_{\a\b} \conv 1_{\b\gamma} = 1_{\a\gamma} 
= 1_{\a\delta} \conv 1_{\delta\gamma},$$ 
so relation \Atwo\  is satisfied in $B$.  The other possibility is that 
only one of $\b$ and $\delta$ lies in $\cP^\gd$, and the other
is in $\cF^\gd \smallsetminus \cP^\gd = \cB \smallsetminus \cP$.  Suppose 
$\b \in \cP$ and $\delta \in \cB\smallsetminus \cP$.
Then $e_\delta = 0$ in $A(\cV^\gd)$, hence the relation \Atwo\  tells us that
$$p(\a, \b)p(\b, \gamma) = p(\a, \delta)p(\delta,\gamma) = 0.$$
On the other hand, the fact that $\delta\in\cB\smallsetminus \cP$ 
implies that $\Delta_{\a}\cap \Delta_{\gamma} = \emptyset$, hence we have
$1_{\a\b} \conv 1_{\a\gamma} = 0$ in $B(\cV)$.

Now suppose that $\a\in\cP^\gd$, $\b\in\cF^\gd$, and $\a \lrao i \b$.
Relation \Athree\  breaks into two cases, depending on whether or not
$\b\in\cP^\gd$.  If $\b\in\cP^\gd$, then we have 
$$\phi(p(\a,\b,\a)) = 1_{\a\b}\conv 1_{\b\a},$$ which is equal to
$u_i \in R_{\a\a} \subset B$ since $S(\a\b\a) = \{i\}$.
In other words, we have $$\phi(p(\a\,b\,a)) = 1_{\a\a}\zeta(u_i) = \phi(e_\a)\phi(t_i),$$
as required.  On the other hand, 
if $\b\in\cF^\gd\smallsetminus\cP^\gd = \cB\smallsetminus\cP$,
then \Athree\ gives the relation $t_ie_\a = 0$ in $A(\cV^\gd)$.
In this case $H_i\cap\Delta_\a = \emptyset$, so $u_i$ goes to $0$ in 
$R_{\a\a}$, and $\phi(t_ie_\a) = 0$.

Thus we have a well-defined homomorphism $$\phi:A(\cV^\gd)\to B(\cV).$$
For each $i\in I$ and $\a\in\cP$ we have $\phi(e_\a t_i) = 1_{\a\a} u_i$, 
which shows that the entire diagonal
subring $\bigoplus_\a R_{\a\a}\subset B$ is contained in the image of $\phi$.  
Surjectivity then follows from the fact that
for any $\b\in\cP$, multiplication by $1_{\a\b} = \phi(p(\a,\b))$
gives the natural quotient map $R_{\a\a} \to R_{\a\b}$.

To show that $\phi$ is injective, we show that 
each block $e_\a A(\cV^\gd) e_\b$ has dimension
no larger than the total dimension of $R_{\a\b}$.
By Proposition \ref{standard form} and Corollary \ref{taut},
we have a surjective map 
$$\chi\colon\R[u_i]_{i\in I}\otimes_{\Sym V}\R = \Sym \R^I/V = \Sym(V^\bot)^*
\to e_\a A(\cV^\gd) e_\b$$
given by substituting $t_i$ for $u_i$ and multiplying by 
any taut path from $\a$ to $\b$.  
It will be enough to show that if $H_S \cap \Delta_\a \cap \Delta_\b = \emptyset$,
then the monomial $u_S$ is in the kernel of $\chi$.

The condition $H_S \cap \Delta_\a \cap \Delta_\b = \emptyset$
can be rephrased as $H_{S'} \cap \Delta_\a = \emptyset$, where $S ' = S \cup \{i \in I \mid \alpha_i \ne \beta_i\}$.
This is equivalent to saying that
the projection $\bar\a$ of $\a$ to $\{\pm 1\}^{I\ssm S'}$ gives
an infeasible sign vector for $\cV^{S'}$.
By Theorem \ref{ftlp} and  Lemma \ref{Gale dual system}, this is
equivalent to saying that $\bar\a$ is 
unbounded for the Gale dual arrangement 
$(\cV^\gd)_{S'}$
(note that $\bar\a$ cannot be infeasible for $(\cV^\gd)_{S'}$, 
since $\Delta_{\bar\a}^\gd \supset \Delta_{\a}^\gd$ and $\a\in\cP=\cP^\gd$).
The vanishing of $\chi(u_S)$ in $A(\cV^\gd)$ then
follows from Corollary \ref{visits every region}.
\end{proof}

\begin{corollary}
We have $\tilde A(\cV^\vee)\cong \tilde B(\cV)$ as graded $\Sym(\R^I/V)$-algebras.
\end{corollary}

\begin{proof}
This follows from Theorem \ref{convolution equals quiver} and the last part of Remark \ref{Bprime}.
\end{proof}

\subsection{The center}\label{the center}
In this section we state and prove a generalization of part (2) of Theorem (B),
which gives a cohomological interpretation of the center of $B$.
Recall the homomorphism $$\zeta\colon \Sym\R^I \to \tilde{B}$$ defined 
in Remark \ref{Bprime}.

\begin{theorem}\label{centerthm}
The image of $\zeta$ is the center of $\tilde{B}$, which is isomorphic to $\tilde R_{\cH}$
as a quotient of $\Sym\R^I$.  The quotient homomorphism
$\tilde{B} \to B$ induces a surjection of centers, and yields an isomorphism
$Z(B) \cong R_{\cH}$.
\end{theorem}

The proof of this theorem goes through several steps.  We define
{\bf extended} rings $\tilde{B}_\ext$ and $B_\ext$ by putting
\[\tilde{B}_\ext := \bigoplus_{(\a,\b)\in \cF \times \cF} \tilde{R}_{\a\b}[-d_{\a\b}]
\,\,\,\text{and}\,\,\, B_\ext := \tilde{B}_\ext \otimes_{\Sym(V)} \R.\]
The difference between these rings and the original ones is that we now use
all feasible sign vectors rather than just the bounded feasible ones.
We define product operations $\conv$ on the extended rings and 
a homomorphism $\zeta^{}_\ext \colon \Sym\R^I \to Z(\tilde{B}_\ext)$
exactly as before, replacing the set $\cP$ with $\cF$.
The topological description or our rings given in Section 
\ref{section:topological B} also carries over, 
replacing the relative core $\eX$ with the extended 
core $\ecore$ (both defined in Section \ref{toric}).
Our strategy will be first to prove Theorem \ref{centerthm} with $\tilde B$
and $B$ replaced by $\tilde B_\ext$
and $B_\ext$, respectively, and then to show that the natural quotient homomorphisms
$\tilde{B}_\ext \to \tilde{B}$ and $B_\ext \to B$ induce isomorphisms of centers.

We begin by constructing a chain complex whose homology is the center $Z(\tilde{B}_\ext)$.
Define the set 
\[\mfD = \{\Delta_\a \cap \Delta_\b \mid \a, \b\in \cF, \Delta_\a\cap \Delta_\b\ne \emptyset\}.\]
It is the set of all faces of chambers of the arrangement $\cH$.  For 
any face $\Delta\in \mfD$, let $\Delta^\circ$ denote its relative interior, that is, its
interior as a subspace of its linear span.  Then $\{\Delta^\circ \mid \Delta\in \mfD\}$
is a decomposition of $V_\eta$ into disjoint cells.
For an integer $d \ge 0$, let $$\mfD_d = \{\Delta \in \mfD \mid \dim(\Delta) = d \}.$$

For a face $\Delta\in \mfD$, its space of orientations is the one-dimensional vector space
\[\ori(\Delta) := H^{BM}_{\dim(\Delta)}(\Delta^\circ; \R).\]
If $\dim(\Delta) = d$, there is a natural boundary map 
\[\bdy_\Delta\colon \ori(\Delta) \to \!\!\!\!\!\bigoplus_{\,\,\,\,\,\,\,\,\Delta\,\supset\,\Sigma \,\in\, \mfD_{d - 1}}\!\!\!\!\! \ori(\Sigma).\]
Putting these together over all $\Delta$ makes $\bigoplus_{\Delta \in \mfD_d} \ori(\Delta)$ into 
a chain complex, graded by the dimension of $\Delta$, which computes the Borel-Moore homology of $V_\eta$. 
Thus its homology
is one-dimensional in degree $\dim V$ and zero in all other degrees.

We next define a chain complex $C_{\udot}$ by putting
\[C_d = \bigoplus_{\Delta \in \mfD_d} \tilde{R}_\Delta \otimes_\R \ori(\Delta),\]
with boundary operator
\[\tilde{R}_\Delta \otimes_\R \ori(\Delta) \to
\!\!\!\!\!\bigoplus_{\,\,\,\,\,\,\,\,\Delta\,\supset\,\Sigma \,\in\, \mfD_{d - 1}} \tilde{R}_{\Sigma} \otimes_\R  \ori(\Sigma) \]
induced by the natural maps $\ori(\Delta) \to \ori(\Sigma)$ and 
$\tilde{R}_{\Delta} \to \tilde{R}_{\Sigma}$ for $\Sigma \subset \Delta$.

Fix an orientation class $\Omega \in H^{BM}_{\dim V}(V; \R)$.
For any $\a \in \cF$, let $\psi_\a\colon \tilde{R}_\cH \to \tilde{R}_{\a}$ denote the natural quotient 
map, and let $\Omega_\a \in \ori(\Delta_\a)$ be the
restriction of $\Omega$ to $\Delta^\circ_\a$.

\begin{lemma}\label{chamber complex}
The complex $C_{\udot}$ has homology only in degree $\dim V$, and we have
an isomorphism $\tilde{R}_\cH \cong H_{\dim V}(C_\udot)$ given by
\[x \mapsto \sum_{\a \in \cF} \psi_\a(x) \otimes \Omega_\a.\]
\end{lemma}

\begin{proof}
Since the terms of $C_\udot$ are direct sums of quotients of $\Sym\R^I$ by 
monomial ideals and all of the entries of the differentials are,
up to sign, induced by the identity map on $\Sym\R^I$, 
this complex splits into a direct sum of complexes of vector spaces,
one for each monomial.  Consider a monomial 
$m = \prod_{i \in S} u_i^{k_i}$, with all $k_i > 0$, 
and let $C_\udot^m \subset C_\udot$ be the subcomplex
consisting of all images of the monomial $m$.
The lemma will follow if we can show that $H_{\dim V}(C^m_\udot)$ is 
a one-dimensional vector space if $H_S \ne \emptyset$ and zero if $H_S = \emptyset$.

We have
\[C_d^m = \bigoplus_{\substack{\Delta \in \mfD_d \\ H_S \cap \Delta
    \ne \emptyset}} \ori(\Delta).\] In particular, if $H_S =
\emptyset$ then $C_\udot^m = 0$.  Assume now that $H_S \ne \emptyset$.
There exists an open tubular neighborhood of $H_S$ in the affine space
$V_\eta$ with the property that for all $\Delta \in \mfD$, 
$\Delta \cap H_S \ne \emptyset$ if and only if $\Delta \cap U \ne \emptyset$.
Then $C_\udot^m$ is the complex computing the cellular Borel-Moore homology of
$U$ using the decomposition by cells $\Delta \cap U$.  It follows that
$H_k(C_\udot^m)$ is one-dimensional if $k = \dim V$ and zero otherwise.
\end{proof}

Using Lemma \ref{chamber complex}, we can deduce the analogue of Theorem
\ref{centerthm} for the extended algebras.

\begin{proposition}\label{centerofBF}
The image of $\zeta_\ext$ is the center of $\tilde{B}_\ext$, which is isomorphic to $\tilde R_{\cH}$
as a quotient of $\Sym\R^I$.  The quotient homomorphism
$\tilde{B}_\ext \to B_\ext$ induces a surjection of centers, and yields an isomorphism
$Z(B_\ext) \cong R_{\cH}$.
\end{proposition}

\begin{proof}
Consider an element $z$ in the center of $\tilde B_\ext$.  Since $z$ commutes with the
idempotent $1_{\a\a}$ for all $\a\in\cF$, 
$z$ must be a sum of diagonal terms, that is, $z = \sum_{\a\in \cF} z_\a$ for some collection
of elements $z_\a\in R_\a = R_{\a\a}$.
For all $\a,\b\in\cF$, let $$\psi_{\a\b}:\tilde R_\a\to \tilde R_{\a\b}$$ be the natural quotient homomorphism.
The fact that $z$ commutes with $1_{\a\b}$ may be translated to the equation
$$\psi_{\a\b}(z_\a) = \psi_{\b\a}(z_\b) \in \tilde R_{\a\b} = \tilde R_{\b\a}.$$ 
On the other hand, since the elements $1_{\a\b}$ for $\a \lra \b$
generate $\tilde B_\ext$ as a ring, these conditions completely characterize $Z(\tilde B_\ext)$.
That is, we have an isomorphism
\begin{equation}\label{center formula}
 	Z(\tilde B_\ext)\cong\left\{(z_\a) \in \bigoplus_{\a \in \cF} \tilde R_\a \bigmid
	\psi_{\a\b}(z_\a) = \psi_{\b\a}(z_\b)\,\,\,
 	\text{for all}\,\,\,\a\lra\b\in \cF\right\}.
 	\end{equation}

Now consider an element
$y = \sum_{\a \in \cF} y_\a \otimes \Omega_\a \in C_{\dim V}$.  
Suppose that $\a, \b \in \cF$ satisfy $\a \lra \b$,
and let $\Omega_{\a\b}$ be the 
orientation of $\Delta_\a \cap \Delta_\b$ induced by $\partial(\Omega_\a)$.
The $\a\b$ component of the differential applied to 
$y$ is $(\psi_{\a\b}(y_a) - \psi_{\b\a}(y_b))\otimes \Omega_{\a\b}$,
thus $(y_\a)$ represents an element of the center
$Z(\tilde{B}_\ext)$ if and only if $y$ is a cycle.  This implies that $\zeta_\ext$ induces an isomorphism
\[\tilde{R}_\cH \cong H_{\dim V}(C_\udot) \cong Z(\tilde{B}_\ext),\]
which proves the first half of the proposition.

Let $\widehat{C}_\udot$ be the complex of 
free $\Sym V$-modules with $\widehat{C}_k = C_k$ for $0 \le k \le \dim V$ and 
$\widehat{C}_{\dim V + 1} = \ker(\partial_{\dim V}) \cong \tilde{R}_\cH$.
We have shown that $\widehat{C}_\udot$ is acyclic, thus so is $\widehat{C}_\udot \otimes_{\Sym(V)} \R$.
Now an argument identical to the one above gives
isomorphisms $$R_\cH \cong H_{\dim V}(\widehat{C}_\udot) \cong Z(B_\ext)$$
compatible with the quotient maps $\tilde R_\cH\to R_\cH$ and $Z(\tilde B_\ext) \to Z(B_\ext)$.
\end{proof}

\begin{remark}\label{other flavors}
The formula \eqref{center formula}
for the center of $\tilde B_\ext$ still holds if we impose the condition
$\psi_{\a\b}(z_\a) = \psi_{\b\a}(z_\b)$ for all 
$\a, \b \in \cF$, regardless of whether or not $\a\lra\b$.
We may re-express this in fancier language by writing
\begin{equation}\label{first inverse limit}
Z(\tilde{B}_\ext) \cong \varprojlim_{\Delta\in \mfD} \tilde{R}_\Delta.
\end{equation}
Identical arguments for $B_\ext$, $\tilde B$, and $B$ give us isomorphisms
\begin{equation}\label{other inverse limits}
Z(B_\ext)  \cong \varprojlim_{\Delta\in \mfD} R_\Delta, \,\,\,\,\,\,\,\,\,
Z(\tilde{B}) \cong \varprojlim_{\Delta\in \mfD(\cP)} \tilde{R}_\Delta, \,\,\,\,\,\,\,\,\,\text{and}\,\,\,\,\,\,
Z(B) \cong \varprojlim_{\Delta\in \mfD(\cP)} R_\Delta,\end{equation}
where
$$\mfD(\cP) := \{\Delta\in \mfD \mid \Delta \subset \Delta_\a \;\text{for some}\;\a \in \cP\}.$$
\end{remark}

For any $\Sigma \in \mfD$, let $\mfD(\Sigma) = \{\Delta
\in \mfD \mid \Delta \subset \Sigma\}$ be the set of its faces, and
let $\mfD^c(\Sigma) = \{\Delta \in \mfD(\Sigma) \mid \Delta\;\text{is
  compact}\,\}$.

\begin{lemma}\label{center lemma}
For any $\Sigma \in \mfD$ and any $\mfD_0$ such that $\mfD^c(\Sigma)\subset \mfD_0 \subset \mfD(\Sigma)$,
the restrictions
\begin{equation}\label{restriction map}
\varprojlim_{\Delta\in \mfD(\Sigma)}\!\! \tilde{R}_\Delta \to\!\! \varprojlim_{\Delta\in \mfD_0}\!\! \tilde{R}_\Delta\;\;\;\;\text{and}\;\;\;\;
  \varprojlim_{\Delta\in \mfD(\Sigma)}\!\! R_\Delta \to \!\!\varprojlim_{\Delta\in \mfD_0}\!\! R_\Delta
\end{equation}
are isomorphisms.
\end{lemma}
\begin{proof}
  If $\Sigma$ is compact (in
  particular if $\dim\Sigma = 0$) the statement is trivial.  So we
  can assume that $\Sigma$ is not compact and, by induction, that the
  statement is true for all proper faces of $\Sigma$.  First we show 
that the lemma holds for $\mfD_0 = \mfD(\Sigma) \ssm \{\Sigma\}$.
Let $C_\udot^\Sigma$ be the subcomplex
  of $C_\udot$ consisting of the summands $\tilde{R}_\Delta
  \otimes_\R \ori(\Delta)$ with $\Delta \in \mfD(\Sigma)$.  
As in the proof of Lemma \ref{chamber complex}, 
 the complex $C_\udot^\Sigma$ splits into a direct sum of
complexes $C^{\Sigma, m}_\udot := C_\udot^\Sigma \cap C_\udot^m$ for each monomial $m$.
The summand $C^{\Sigma, m}_\udot$ is a cellular complex computing the Borel-Moore
homology of a tubular neighborhood of $H_S\cap\Sigma$ in $\Sigma$, where $S$ is the support of $m$.
Since $\Sigma$ is itself non-compact, such a neighborhood (when nonempty) is always homeomorphic
to a non-compact polyhedron with at least one vertex, and therefore has trivial Borel-Moore homology.
It follows that each $C^{\Sigma, m}_\udot$ is acyclic, and thus so is
$C_\udot^\Sigma$.

The fact that the first map of \eqref{restriction map} is an isomorphism for $\mfD_0 = \mfD(\Sigma) \ssm \{\Sigma\}$
now follows from the fact that 
the target is isomorphic to the kernel of the boundary map 
$C^\Sigma_{\dim V - 1} \to C^\Sigma_{\dim V - 2}$.
The second isomorphism follows analogously,
since $C^\Sigma_\udot$ is an acyclic complex of free
$\Sym(V)$-modules, which implies that $C^\Sigma_\udot\otimes_{\Sym(V)} \R$
is an acyclic complex of vector spaces.

Finally, to prove the Lemma for a general $\mfD_0$ containing $\mfD^c(\Sigma)$,
pick an ordering $\Delta_1, \dots, \Delta_r$ of the faces in
$\mfD(\Sigma) \ssm \mfD_0$ so that their dimension is 
nonincreasing, and let $\mfD_j = \mfD_0 \cup \{\Delta_1, \dots, \Delta_j\}$.  
Then for $1 \le j \le r$ all of the proper faces of any $\Delta_j$ already lie in $\mfD_{j-1}$, 
so an argument identical to the one above shows that
\[ \varprojlim_{\Delta\in \mfD_j}\! \tilde{R}_\Delta \to\!\! \varprojlim_{\Delta\in \mfD_{j-1}}\!\!\! \tilde{R}_\Delta \;\;\text{and}\;\;
\varprojlim_{\Delta\in \mfD_j} \!R_\Delta \to\!\! \varprojlim_{\Delta\in \mfD_{j-1}}\!\!\! R_\Delta\]
are isomorphisms.
\end{proof}

We are now ready to prove Theorem \ref{centerthm}.

\begin{proof}[Proof of Theorem \ref{centerthm}]
By Equations \eqref{first inverse limit} and \eqref{other inverse limits} and Lemma \ref{center lemma}, we have
$$Z(\tilde B_\ext) \,\,\,\cong\,\,\, \varprojlim_{\Delta\in \mfD} \tilde{R}_\Delta
\,\,\,\cong\,\,\, \varprojlim_{\Delta\in \mfD(\cP)} \tilde{R}_\Delta
\,\,\,\cong\,\,\, Z(\tilde B)$$
and 
$$Z(B_\ext) \,\,\,\cong\,\,\, \varprojlim_{\Delta\in \mfD} {R}_\Delta
\,\,\,\cong\,\,\, \varprojlim_{\Delta\in \mfD(\cP)} {R}_\Delta
\,\,\,\cong\,\,\, Z(B).$$
Since all of these isomorphisms fit into a commutative diagram
\[\xymatrix{
\Sym\R^I \ar[d]^{=}\ar[r]^{\zeta_\ext} & Z(\tilde B_\ext) \ar[d]^{\cong}\ar[r] & Z(B_\ext)\ar[d]^{\cong} \\
\Sym\R^I\ar[r]^{\zeta} & Z(\tilde B)\ar[r] & Z(B),  
}\]
the theorem is proved.
\end{proof}

\section{The representation category}
\label{sec:repr-categ}
We begin with a general discussion of highest weight categories, quasi-hereditary algebras, self-dual projectives, and Koszul algebras.
With the background in place, we analyze our algebras $A(\cV)$ and $B(\cV)$ 
in light of these definitions.

\subsection{Highest weight categories}
\label{sec:high-weight-categ}

Let $\mathcal C$ be an abelian, artinian category 
enriched over $\R$ with simple
objects $\{\Si \a\mid\a\in\cI\}$,
projective covers $\{\Pro \a\mid\a \in \cI\}$, 
and injective hulls $\{I_\a\mid\a\in\cI\}$.
Let $\leq$ be a partial order on the index set $\cI$.

\begin{definition}\label{high-weight-def}
  We call $\mathcal C$ {\bf highest weight} with respect to this partial order if there is a collection of objects
  $\{\Ve\a\mid\a \in \cI\}$ and epimorphisms 
  $\Pro \a\overset{\Pi_\a}\to\Ve \a\overset{\pi_\a}\to\Si \a$
  such that  for each $\a\in \cI$, the following conditions hold:
  \begin{enumerate}
  \item The object $\ker\pi_\a$ has a filtration such that each sub-quotient is isomorphic to $\Si \b$ for some $\b< \a$.
  \item The object $\ker \Pi_\a$ has a filtration such that each sub-quotient is isomorphic to $\Ve \ga$ for some $\ga> \a$.
  \end{enumerate}
The objects $\Ve\a$ are called {\bf standard objects}.
Classic examples of highest weight categories in representation theory include the various integral
blocks of parabolic category $\cO$ \cite[5.1]{FM}.
\end{definition}

Suppose that $\cC$ is highest weight with respect to a given partial order on $\cI$.
To simplify the discussion, we will assume that the endomorphism algebras of 
every simple object in $\cC$ is just the scalar ring $\R$; this will hold
for the categories we consider.
For all $\a\in\cI$, let
$\mathcal C_{\not> \a}$ be the subcategory of objects 
whose composition series contain no simple objects $\Si\b$ with $\b >\a$.
By \cite[3.2(b)]{CPS}, the standard object $\Ve\a$ is isomorphic to the projective cover of $\Si\a$ in the
subcategory $\mathcal C_{\not> \a}$.
Dually, we define the {\bf costandard object} $\Lambda_\a$ to be the
the injective hull of $L_\a$ in $\mathcal C_{\not> \a}$.

\begin{definition}\label{tilting}
An object of $\cC$ is called {\bf tilting} if it admits a filtration with standard sub-quotients
{\em and} one with costandard sub-quotients.
An equivalent condition is that $T$ is tilting if and only if $\Ext^i(T,\Lambda_\a)=0=\Ext^i(\Ve\a,T)$ for all $i> 0$ and $\a\in\cI$.
(The first condition is equivalent to the existence of a
standard filtration, and the second to the existence of a costandard filtration.)
For each $\a\in\cI$, there is a unique indecomposible tilting module $\Tilt\a$
with $\Ve\a$ as its largest standard submodule and $\Lambda_\a$ 
as its largest costandard quotient \cite{Ri}.
\end{definition}

We now have six important sets of objects of $\cC$, all indexed by the set $\cI$:
\begin{itemize}
\item the simples $\{\Si\a\}$
\item the indecomposable projectives $\{\Pro\a\}$
\item The indecomposable injectives $\{I_\a\}$
\item the standard objects $\{\Ve\a\}$
\item the costandard objects $\{\Lambda_\a\}$
\item the tilting objects $\{T_\a\}$.
\end{itemize}
Each of these six sets forms a basis for the Grothendieck group $K(\cC)$, and
thus each is a minimal set of generators of the bounded derived category
$D^b(\mathcal{C})$.  In particular, any exact functor from
$D^b(\mathcal C)$ to any other triangulated category is determined by
the images of these objects and the morphisms between them and their 
shifts.

Let $\{M_\a\mid \a\in\cI\}$ and $\{N_\a\mid \a\in\cI\}$
be two sets of objects that form bases for $K(\cC)$.
We say that the second set is {\bf left dual} to the first set
(and that the first set is {\bf right dual} to the second) if 
\begin{equation*}
 \Ext^i\left(N_\a,M_{\b}\right)\cong
\begin{cases}
 \R & \text{if $\a=\b$ and $i=0$,}\\
0 & \text{otherwise}.
\end{cases}
\end{equation*}
It is an easy exercise to check that if a dual set to $\{M_\a\}$ exists, then it is unique
up to isomorphism.
Note that dual sets descend to dual bases for $K(\cC)$ under the Euler form 
$$\big\langle [M],[N]\big\rangle := \sum_{i=0}^\infty (-1)^i\dim\Ext^i(M,N).$$

\begin{proposition}\label{dual sequences}
  The sets $\{P_\a\}$ and $\{I_\a\}$ are left and right (respectively) dual to $\{L_\a\}$, and 
  the set $\{\Lambda_\a\}$ is right dual to $\{\Ve\a\}$.
\end{proposition}
\begin{proof}
  The first statement follows from the definition of projective covers and
  injective hulls.  The second statement is shown in the proof of \cite[3.11]{CPS}. 
  \end{proof}

\subsection{Quasi-hereditary algebras}\label{qha}
We now study those algebras whose module categories are highest weight.

\begin{definition} An algebra is {\bf quasi-hereditary} if its category $\cC(E)$
of finitely generated right modules is highest weight with respect to some partial ordering of its simple modules.
\end{definition}

Let $E$ be a finite-dimensional, quasi-hereditary $\R$-algebra with respect to a fixed partial order
on the indexing set $\cI$ of its simple modules.
Let $$\Pro* = \bigoplus_{\a\in\cI} P_\a,\,\,\, I_* = \bigoplus_{\a\in\cI} I_\a,
\,\,\,\text{and}\,\,\,\Tilt* = \bigoplus_{\a\in\cI} T_\a$$ 
be the sums of the indecomposible projectives, injectives, and tilting modules,
respectively.  Let $D^b(E) = D^b(\cC(E))$ be the bounded derived category of finitely generated right $E$-modules.

\begin{definition}  We say that $E$ is {\bf basic} if the simple module $L_\a$
is one-dimensional for all $\a\in\cI$.  This is equivalent to requiring that
the canonical homomorphism
$$E\to\End(\Pro*)\cong\End(I_*)^{\text{op}}$$
is an isomorphism.
\end{definition}

\begin{definition}
The endomorphism algebra $R(E) := \End(\Tilt*)$ is called the {\bf Ringel dual} of $E$.  
It has simple modules indexed by $\cI$, and it is quasi-hereditary with respect to the
partial order on $\cI$ opposite to the given one.
If $E$ is basic, then the canonical homomorphism $E\to R(R(E))$ is an isomorphism
\cite[Theorems 6 \&\ 7]{Ri}.
The functor $\EuScript{R} := \RHom^\bullet(-,\Tilt*)$
from $D^b(E)$ to $D^b(R(E))$ is called the {\bf Ringel duality functor}.
\end{definition}

\begin{proposition}\label{ringdu}
Suppose that $E$ is basic.
  Up to automorphisms of $E$ and $R(E)$, $\EuScript{R}$ is the unique
  contravariant equivalence that satisfies any of the following
  conditions:
  \begin{enumerate}
   \item $\EuScript{R}$ sends tilting modules to projective modules,
 \item $\EuScript{R}$ sends projective modules to tilting modules,
  \item $\EuScript{R}$ sends standard modules to standard modules.
  \end{enumerate}
\end{proposition}
\begin{proof}
  We first prove statement (1).  The Ringel duality functor is an
  equivalence because $D^b(E)$ is generated by $T_*$.
  Since $\EuScript{R}(T_*)$ is equal to $R(E)$ as a right module over itself,
  it is clear that $\EuScript{R}$ takes tilting modules to projective modules.
  Suppose that $\EuScript{R}'$
  is another such equivalence.  Since the indecomposable tilting
  modules $\{T_\a\}$ generate $K(E)$ and $\EuScript{R}'$ induces an
  isomorphism on Grothendieck groups, $\EuScript{R}'$ must take the
  tilting modules to the complete set of indecomposable projective
  $R(E)$-modules.  Thus $\EuScript{R}'(T_*)$ is isomorphic to the
  direct sum of all indecomposable projective $R(E)$-modules, which is
  isomorphic to $R(E)$ as a right module over itself.  Since any exact
  functor is determined by its values on sends a generator and on the
  endomorphisms of that generator, $\EuScript{R}'$ can only differ
  from $\EuScript{R}$ in its isomorphism between $\End(T_*)$ and
  $R(E)$.  This is precisely the uniqueness statement we have claimed
  for (1).

  Statement (2) follows by applying statement (1) to the adjoint
  functor.  

As for Statement (3), it was shown in \cite[Theorem 6]{Ri} that
  $\EuScript{R}$ takes standard modules to standard modules.
  Suppose that $\EuScript{R}'$ is another such equivalence.
  For any $\a\in\cI$, the projective module $P_\a$ has a standard filtration,
  therefore so does $\EuScript{R}'(P_\a)$.  Furthermore, we have 
  $\Ext^i(\EuScript{R}'(\Ve\b),\EuScript{R}'(\Pro\a))=\Ext^i(\Pro\a,\Ve\b)=0$
  for all $i>0$ and $\b\in\cI$, thus $\EuScript{R}'(\Pro\a)$ 
  has a costandard filtration as well, and is therefore tilting.  
  Then part (2) tells us that $\EuScript{R}'$ is the Ringel duality functor.
\end{proof}

\subsection{Self-dual projectives and the double centralizer property}\label{sdp}
Suppose that our algebra $E$ is basic and 
quasi-hereditary, and that it is
endowed with an
anti-involution $\psi$, inducing an equivalence of categories $\cC(E)\simeq \cC(E^{\text{op}})$.
We have another such equivalence given by taking the dual of the underlying vector space,
and these two equivalences compose to a contravariant 
auto-involution $\du$ of $\cC(E)$.

We will assume for simplicity that $\psi$ fixes all idempotents of $E$.  The case where it gives a non-trivial involution on idempotents is also interesting, but requires a bit more care in the statements below, and will not be relevant to this paper.
The following proposition follows easily from the fact that any contravariant equivalence takes projectives to injectives.

\begin{proposition}\label{duality}
For all $\a\in\cI$, 
$$\du \!L_\a\cong L_\a,\,\,\,\,\,\du\!\Pro\a\cong I_\a,\,\,\,\,\,\du\!\Ve\a\cong\Lambda_\a,\,\,\,\,\,\text{and}\,\,\,\,\,\du\!\Tilt\a\cong \Tilt\a.$$
\end{proposition}

\begin{remark}\label{consequences}
Proposition \ref{duality} has two important consequences.  First, since $\du$ preserves simples,
it acts trivially on the Grothendieck group of $\cC(E)$.  In particular, we have $[\Ve \a]=[\Lambda_\a]$, so by Proposition \ref{dual sequences}, the classes $[\Ve \a]$ are an orthonormal basis of the Grothendieck group.

Second, the isomorphism $T_*\cong\du\!T_*$ induces an anti-automorphism of $R(E)$ that fixes idempotents, 
and thus a duality functor on the Ringel dual category $\cC(R(E))$.
\end{remark}

The next proposition follows immediately from the definitions and the fact that all
tilting modules are self-dual.

\begin{proposition}\label{intisd}
For all $\a\in\cI$, the following are equivalent:
  \begin{enumerate}
  \item The projective module $\Pro \a$ is injective.
   \item The projective module $\Pro \a$ is tilting.
  \item The projective module $\Pro \a$ is self-dual, that is, $\du(\Pro\a) = \Pro\a$.
  \end{enumerate}
\end{proposition}

We will later need the following easy lemma.

\begin{lemma}\label{general-self-dual}
    If $\Pro\a$ is self-dual, then
the simple module $\Si\a$ is contained in the socle of some standard module $\Ve \b$.
\end{lemma}

\begin{proof}


  Suppose that the projective module $\Pro\a$ is self-dual.
  Since $\Pro \a$ is indecomposible, it is the injective hull
  of its socle.  Since $\Pro\a$ is self-dual, its socle is isomorphic
  to its cosocle $\Si \a$.  Since $P_\a$ has a standard filtration, it has at least one
  standard module $\Ve\b$ as a submodule.
  The functor that takes a module to its socle is
  left exact and $\Ve\b$ is finite-dimensional (and therefore has a non-trivial socle),
  hence the socle of $\Ve\b$ is a non-trivial submodule of the socle $L_\a$
  of $\Pro\a$.  Since $L_\a$ is simple, the socle of $\Ve\b$ must be isomorphic to $L_\a$.
  \end{proof}

Let $\cI_{\du} = \{\a\in\cI\mid\du(\Pro\a) \cong \Pro\a\}$.  Let
$$\PI  = \bigoplus_{\a\,\in\,\cI_{\du}}\Pro \a\subset P_*
$$ be the direct sum of all of the
self-dual projective right $E$-modules, and consider its endomorphism
algebra
\begin{equation}\label{erect S}\PIS := \End(\PI)\subset\End(P_*) \cong E.
\end{equation}

\begin{definition}\label{symdef}
An algebra is said to be {\bf symmetric} if it is isomorphic to its
vector space dual as a bimodule over itself.
It is immediate from the definition that $\PIS$ is symmetric.
\end{definition}

The next theorem, which we will need in Section \ref{Ringel duality and Serre functors}, 
provides a motivation for studying self-dual projectives and their endomorphism algebras.

\begin{theorem}\label{dc}
Suppose that the converse of Lemma \ref{general-self-dual} holds for the algebra $E$.
Then the functor from right $E$-modules to right $\PIS$-modules taking a
  module $M$ to $\Hom_E(\PI,M)$ is fully faithful on projectives.
\end{theorem}

\begin{proof}
  Fix an index $\a\in\cI_{\du}$, and let $L$ be the socle of the
  standard module $\Ve \a$.  Then $L$ is a direct sum of simple modules,
  and the assumption above implies that the injective hull of $L$ is
  also projective; we denote this hull by $P$.  Since $P$ is an
  injective module, the inclusion $L \hookrightarrow P$ extends to
  $\Ve\a$, and since the map is injective on the socle $L$, it must be
  injective on all of $\Ve\a$.  Thus $P$ is the injective hull of
  $\Ve\a$.  An application of \cite[2.6]{MS} gives the desired
  result.
\end{proof}

\begin{remark}\label{dcp}
  The property attributed to the $\PIS$-$E$-bimodule $\PI$ in Theorem
  \ref{dc} is known as the {\bf double centralizer property}.  See
  \cite{MS} for a more detailed treatment of this phenomenon.
\end{remark}

\subsection{Koszul algebras}\label{Koszul algebras}
To discuss the notion of Koszulity, we must begin to work with graded algebras and graded modules.
Let $E = \bigoplus_{k\geq 0}E_k$ be a graded $\R$-algebra, and let $R = E_0$. 

\begin{definition}
  A complex
  \[
  	\hdots\to P_k \rightarrow P_{k-1} \rightarrow \hdots\to P_1 \rightarrow P_0
\]
of graded projective right $E$-modules is called {\bf linear} if $P_k$ is generated in
  degree $k$.
\end{definition}

\begin{definition}
The algebra $E$ is called {\bf Koszul} if each every simple right $E$-module admits a linear projective resolution.
\end{definition}

The notion of Koszulity gives us a second interpretation of quadratic duality.

\begin{theorem}\label{kd}{\cite[2.3.3, 2.9.1, 2.10.1]{BGS96}}
If $E$ is Koszul, then it is quadratic.  Its quadratic dual $E^!$ is also Koszul, and is isomorphic to
$Ext_E(R,R)^{\text{op}}$.
\end{theorem}

\begin{remark}
In this case $E^!$ is also known as the {\bf Koszul dual} of $E$.
\end{remark}

Let $D(E)$ be the bounded derived category of graded
right $E$-modules.

\begin{theorem}{\cite[1.2.6]{BGS96}}\label{koszul-equiv}
  If $E$ is Koszul, we have an equivalence of derived categories 
  \begin{equation*}
    D(E)\cong D(E^!).
  \end{equation*}
\end{theorem}


We conclude this section with a discussion of Koszulity for
quasi-hereditary algebras.  Suppose that our graded algebra $E$ is
quasi-hereditary.  For all $\a\in\cI$, there exists an idempotent
$e_\a\in R$ such that $P_\a = e_\a E$, thus each projective module
$P_\a$ inherits a natural grading.  Let us assume that the grading of
$E$ is compatible with the quasi-hereditary structure.  In other
words, we suppose that for all $\a\in\cI$, the standard module $V_\a$
admits a grading that is compatible with the map
$\Pi_\a:\Pro\a\to\Ve\a$ of Definition \ref{high-weight-def}.  It is
not hard to check that each of $\Si\a,\Ve\a,\Lambda_\a,\Pro\a$, and
$I_\a$ inherits a grading as a quotient of $E$, and that $\Tilt\a$
admits a unique grading that is compatible with the inclusion of
$\Ve\a$.  Thus $R(E) = \End(T_*)$ inherits a grading as well, and this
grading is compatible with the quasi-hereditary structure \cite{Zhu}.

\begin{theorem}\label{ADL}{\cite[1]{ADL03}}
Let $E$ be a finite-dimensional graded algebra with a graded
anti-automorphism that preserves idempotents. 
If $E$ is graded quasi-hereditary and each standard module admits
a linear projective resolution, then $E$ is Koszul.
\end{theorem}

\subsection{The algebra \boldmath{$A(\cV)$}}
Let $\cV = (V,\eta,\xi)$ be a polarized arrangement, and let $A = A(\cV)$ be the associated quiver algebra.  
$A$ has a canonical anti-automorphism taking
$p(\a,\b)$ to $p(\b,\a)$ for all $\a\lra\b$ in $\cP$.
Under the identification $$A(\cV)\,\, \cong\,\, B(\cV^\gd) \,\, = \bigoplus_{(\a,\b)\,\in\,\cP\times\cP}R^\vee_{\a\b}[-d^\vee_{\a\b}]$$
of Theorem \ref{convolution equals quiver},
this corresponds identifying
$R^\vee_{\a\b}$ with $R^\vee_{\b\a}$.  Geometrically, it is given by
swapping the left and right factors of the fiber product
$\wt{\eX}^\gd\times_{\eX^\gd}\wt{\eX}^\gd$.
This anti-automorphism fixes the idempotents, and thus gives rise to a contravariant
involution $\du$ of $\cC(A)$ as in Section \ref{sdp}.

For all $\a\in\cP$, let 
$$L_\a = A/\langle e_\b\mid \b\neq\a\rangle$$ be the one-dimensional simple right $A$-module
supported at the node $\a$,
and let $P_\a = e_\a A$ denote the projective cover of $L_\a$.
Since $L_\a$ is one-dimensional for each $\a$, $A$ is basic.
Let $a = \mu^{-1}(\a)$ be the basis corresponding to the sign vector $\a$,
and let   \begin{equation*}
    K_{> \a}=\sum_{i\in a} p(\a,\a^i)\cdot A
    \subset \Pro \a
  \end{equation*}
be the right-submodule of $P_\a$ generated by paths that begin at the node $\a$ and move
to a node that is higher in the partial order given in Section \ref{sec:partial order}.  (Recall that $\a_i$ is the unique sign vector
such that $\a\lrao i\a^i$.)  Let $$V_\a = P_\a/K_{>\a},$$ and let
$$P_\a\to V_\a\to L_\a$$ be the natural projections.  

\begin{lemma}\label{vsb}
The module $\Ve\a$ has a vector space basis consisting of a taut path
from $\a$ to each element of $\cF\cap\cB_a$.
\end{lemma}

\begin{proof}
Corollary \ref{taut} implies that such a collection of paths is linearly independent.  
Any taut path which terminates outside of $\cF\cap\cB_a$ must cross a hyperplane
$H_i$ for some $i\in a$, and by Corollary \ref{visits every region} it can be replaced by a path which
crosses this hyperplane first, thus it lies in $K_{>\a}$.  It will
therefore suffice to show that any
path which is not taut will also have trivial image in $\Ve\a$.
By Proposition \ref{standard form}, this is equivalent to showing that the positive degree part of $\Sym V$
acts trivially on $\Ve\a$, which follows from the fact that $V$ is spanned by $\{t_i\mid i\in a\}$.
\end{proof}

When $\cV$ is rational, the modules 
$\Pro \a ,\Ve \a, \Si \a$ acquire natural geometric interpretations
via the isomorphisms $A \cong B(\cV^\gd) \cong H^*(\wt\eX^\gd \times_{\eX^\gd} \wt\eX^\gd)$ given by 
Theorem \ref{convolution equals quiver} and the results of Section \ref{section:topological B}.
For each $\a\in\cP=\cP^\gd$, we have a relative core component $X_\a^\gd\subs\eX^\gd$.
Let $y_\a\in X_\a^\gd$ be an arbitrary element of the dense toric stratum (in other words,
an element whose image under the moment map lies in the interior of the polyhedron $\Delta_\a^\gd$),
and let $x_\a = \lim_{\la\to\infty} \la\cdot y_\a \in X^\gd_\a$ be the toric fixed point whose image under the moment
map is the $\xi^\gd$-maximum point of $\Delta_\a^\gd$.

\begin{proposition}\label{interpretations}
If $\cV$ is rational, then for any $\a \in \cP$ we have module isomorphisms
$$\Pro \a\cong H^*(X_\a\times_{\eX^\gd}\wt{\eX}^\gd),\,\,\,\, 
  \Ve \a\cong H^*(\{x_\a\}\times_{\eX^\gd}\wt{\eX}^\gd),\,\,\,\,\text{and}\,\,\,\,
  \Si \a\cong H^*(\{y_\a\}\times_{\eX^\gd}\wt{\eX}^\gd),$$
  where $A(\cV) \cong B(\cV^\gd) \cong H^*(\wt{\eX}^\gd\times_{\eX^\gd}\wt{\eX}^\gd)$ acts on the right by convolution.
\end{proposition}

\begin{proof}
The first isomorphism is immediate from the definitions.  

Restriction to the point $x_\a$ defines a surjection
$$\Pro \a\cong H^*(X^\gd_\a\times_{\eX^\gd}\wt{\eX}^\gd)\to H^*(\{x_\a\}\times_{\eX^\gd}\wt{\eX}^\gd).$$
Note that $x_\a\in X^\gd_\b$ if and only if
$\a(i) = \b(i)$ for all $i \notin a := \mu^{-1}(\a)$, or in other words, if and
only if $\b \in \cF_a \cap \cB = \cB^\gd_a \cap \cF^\gd$.
The second isomorphism now follows from Lemma \ref{vsb}, 
using the fact that a taut path from $\a$ to $\b$ in the algebra 
$A(\cV)$ gives rise to the unit class
 $1_{\a\b} \in H^0(X_{\a\b}) \subset B(\cV^\gd)$.

The third isomorphism follows from the fact that 
$H^*(\{y_\a\}\times_{\eX^\gd}\wt{\eX}^\gd)\cong H^*(\{y_\a\})$ is one-dimensional, $e_\a$ acts by the identity,
and $e_\b$ acts by zero for all $\b\neq\a$.
\end{proof}

\begin{theorem}\label{CA-highest-weight}
 The algebra $A$ is quasi-hereditary with respect to the partial order on $\cP$ given in Section \ref{sec:partial order},
 with the modules $\{V_\a\}$ as the standard modules.  This structure is compatible with the grading on $A$.
 \end{theorem}
 
\begin{proof}
  We must show that the 
  modules $\{\Ve \a\mid\a \in \cP\}$
  satisfy the conditions of Definition \ref{high-weight-def}.
  Condition (1) follows from Lemmas \ref{vsb} and \ref{order lemma}.

  For condition (2), we define a filtration of $K_{>\a} = \ker \Pi_\a$ as follows.  
  For $\gamma \in \cP$, let $P_\a^\gamma\subset \Pro \a$ be the submodule generated 
  by paths that pass through the node $\gamma$, and for any $\gamma \in \cP$ let 
  $$P_\a^{\ge \gamma} = \sum_{\delta\ge\gamma} P_\a^\delta\,\,\,\text{and}\,\,\,
  P_\a^{>\gamma} = \sum_{\delta>\gamma} P_\a^\delta.$$
  Note that $P_\a^{\ge \a} = P_\a$ and $P_\a^{>\a} = K_{>\a}$.

  Then $P_\a^{\gamma} \subset K_{>\a}$ for all $\gamma>\a$,
  and these submodules form a filtration with
  sub-quotients
  $$M_\a^\gamma:=P_\a^{\ge\gamma}/P_\a^{> \gamma}.$$  
  Let $g = \mu^{-1}(\gamma)\in\Bas$.
If $\a$ is not in the negative cone $\cB_{g}$, then there exists
$i\in g$ such that $\a(i) \neq \gamma(i)$.  It follows from Corollary \ref{taut}
that $P_\a^\ga = P_\a^{>\ga}$, hence $M_\a^\ga = 0$.  
If $\a\in\cB_{g}$, then composition with a taut path $p$ from $\a$ to $\ga$ defines a
map $P_\ga \to P_\a^{\ge \ga}$ which induces a map
$V_\gamma \to M_\a^\gamma$.  We will show that this induced map is an isomorphism.
First, note that $M_\a^\ga$ is spanned by the classes of paths which
pass through $\gamma$.  Using Proposition \ref{standard form}, such a
path is equivalent to one which begins with a taut path from $\a$
to $\gamma$, and by Corollary \ref{taut} implies that this 
taut path can be taken to be $p$, so our map is surjective.

To see that it is injective, it will be enough to show that
\[\dim P_\a  = \sum_{\a \in \cB_g} \dim V_{\mu(g)} = |\{(\delta, g)\in \cP\times \Bas \mid \a, \delta \in \cB_g\}|.\]
Since surjectivity establishes one inequality, it is sufficient to show that we have equality when we
sum over all $\a$, that is, that
\[\dim A = |\{(\a, \b, g)\in \cP\times \cP \times \Bas\mid \a, \b \in \cB_g\}|.\]
By Theorem \ref{convolution equals quiver} and Lemma \ref{formality} we have
\begin{equation*}
 \dim A = \sum_{\alpha, \b \in \cP^\gd} \dim R^\gd_{\a\b}
 = |\{(\a,\b, g)\in \cP \times \cP \times \Bas \mid H^\gd_{g^c} \subset \Delta^\gd_\a \cap \Delta^\gd_\b\}|.
\end{equation*}
Finally, we observe that for any basis $g \in \Bas$, 
$$H^\gd_{g^c} \subset \Delta^\gd_\a \iff \a \in \cF^\vee_{g^c} \iff \a \in \cB_{g},$$
and the result follows.  
\end{proof}

\begin{theorem}\label{Koszul dual} 
Let $\cV$ be a polarized arrangement.  The algebras $A(\cV)$ and $B(\cV)$ are Koszul,
and Koszul dual to each other.
\end{theorem}

\begin{proof}
By Theorems \ref{Quadratic dual}, \ref{convolution equals quiver}, and \ref{kd}, it is enough to prove that
$A = A(\cV)$ is Koszul.
By Theorem \ref{ADL}, it is enough to show that each standard module $\Ve\a$
has a linear projective resolution.

Let $a = \mu^{-1}(\a)$ be the basis associated with the sign vector
$\a$.  For any subset $S\subset a$, let $\a^S$ be the sign vector that
differs from $\a$ in exactly the indices in $S$.  Thus, for example,
$\a^\emptyset = \a$, and $\a^{\{i\}} = \a^i$ for all $i\in a$.  
(Note that the sign vectors that arise this way are exactly those in the set $\cF_a$.)
If $S = S' \sqcup \{i\} \subset a$, then we have a map
$\vp_{S,i}:P_{\a^{S}}\to P_{\a^{S'}}$ given by left multiplication by
the element $p(\a^{S'}, \a^{S})$.  We adopt the convention that
$P_{\a^S} = 0$ if $\a^S\notin\cP$ and $\vp_{S,i}=0$ if $i\notin S$.

Let $$\Pi_\a=\bigoplus_{S\subset a}\Pro{\a^S}$$ be the sum of all of
the projective modules associated to the sign vectors $\a^S$.
This module is multi-graded by the abelian group $\Z^a = \Z\{\ep_i\mid i\in a\}$,
with the summand $\Pro{\a^S}$ sitting in multi-degree $\ep_S = \sum_{i\in S}\ep_i$.
For each $i\in a$, we define a differential 
$$\partial_i = \sum_{i\in S\subset a}\vp_{S,i}$$ of degree $-\ep_i$.
These differentials commute because of the relation (A2), and thus
define a multi-complex structure on $\Pi_\a$.  The total complex $\tPi$
of this multi-complex is linear and projective; we claim that it is a resolution
of the standard module $\Ve\a$.
It is clear from the definition that $H^0(\tPi)\cong \Ve\a$, thus we need only show that our complex is 
exact in positive degrees.

We will use two important facts about filtered chain complexes and multi-complexes.  
Both are manifest from the theory of spectral sequences, but could also easily be proven by hand by any interested reader.
\begin{itemize}
\item[(*)] If any one of the differentials in a multi-complex is exact, then the total complex 
is exact as well.
\item[(**)] If a chain complex $\mathbf{C}^\bullet$ has a filtration such that the associated 
graded $\wt{\mathbf{C}}^\bullet$ is exact at an index $i$,
then $\mathbf{C}^\bullet$ is also exact at $i$.
\end{itemize}

As in the proof of Theorem~\ref{CA-highest-weight},
we may filter each projective module $P_{\a^S}$ by submodules
of the form $P_{\a^S}^\b$ for $\b\geq\a^S$, 
which consists of paths from $\a_S$ that pass through the node $\b$.
We extend this filtration to all $\beta$ by defining
$P^\b_{\a^S}$ to be the sum of $P^{\b'}_{\a^S}$
over all $\b' \in \cP$ for which $\b' \ge \a^S$
and $\b' \ge \b$. 
It is easy to see that this filtration is compatible with the differentials,
hence we obtain an associated graded multi-complex
$$\widetilde\Pi_\a^\bullet := \bigoplus_{\b}\,\,(\Pi_\a^\bullet)^{\b}/(\Pi_\a^\bullet)^{>\b}
= \bigoplus_{\b,S}M_{\a^S}^\b.$$

Take $\b \in \cP$, and let $b = \mu^{-1}(\b)$.
We showed in the proof of Theorem \ref{CA-highest-weight} that
$M_{\a^S}^\b$ is non-zero if and only if
$\a^S\in\cB_b$, in which case it is isomorphic to $\Ve\b$.
If $\b = \a$, then we have a non-zero summand only when $S = \emptyset$,
so that summand sits in total degree zero.  For $\b\neq\a$, choose an element
$i\in a\cap b^c$.  This ensures that if $S = S'\sqcup\{i\}$, then $\a^{S'}\in\cB_b$
if and only if $\a^S\in\cB_b$.  For such a pair $S$ and $S'$,
we have $$M^\b_{\a^S}\cong\Ve\b\cong M^\b_{\a^{S'}},$$
and $\tilde\partial_i^\b$ is the isomorphism given by left-composition with
$p(\a^{S'},\a^S)$.  Thus $\tilde\partial_i^\b$ is exact in non-zero degree.
By (*) we can conclude that the total complex $\wt\tPi$ is exact in non-zero degree,
and thus by (**) so is $\tPi$.
\end{proof}

We next determine which projective $A$-modules are self-dual.

\begin{theorem}\label{self-dual}
  For all $\a\in\cP$, the following are equivalent:
  \begin{enumerate}
  \item The projective $\Pro \a$ is injective.
   \item The projective $\Pro \a$ is tilting.
  \item The projective $\Pro \a$ is self-dual, that is, $\du(\Pro\a) = \Pro\a$.
  \item The simple $\Si\a$ is contained in the socle of some standard module $\Ve \b$.
 \item The cone $\Sigma_\a\subset V$ has non-trivial interior.
 \item The chamber $\Delta_\a^\gd\subset V^\perp_{-\xi}$ is compact.
   \end{enumerate}
\end{theorem}
\begin{proof}
The implications $(1)\Leftrightarrow(2)\Leftrightarrow(3)$ were proved in Proposition \ref{intisd}.
The fact that any of these implies $(4)$ was proven in Lemma \ref{general-self-dual}.

  $(4)\Rightarrow(5)$:  Let $b = \mu^{-1}(\b)\in\Bas$.  By Lemma \ref{vsb}, $\Ve\b$ is spanned as a vector space
  by taut paths $p_\gamma$ from $\b$ to nodes $\gamma\in\cF\cap\cB_b$.  The socle
  of $\Ve\b$ is spanned by those $p_\gamma$ for which $\gamma$ is as
far away from $\b$ as possible.  More precisely, if $i\notin b$ and $H_i$ 
meets $\Delta_\gamma$, then $H_i$ must separate $\Delta_\gamma$ from $\Delta_\b$.
This implies that any ray starting at the point $H_b$ and passing
through an interior point $q$ of $\Delta_\gamma$ will not leave this
chamber once it enters.  It follows that the direction vector of this
ray lies in $\Sigma_\gamma$.  Since this holds for any $q$, $\Sigma_\gamma$
has nonempty interior.

$(5)\Rightarrow(6)$: The fact that $\Sigma_\a$ has non-empty interior implies that 
$\a$ is feasible for the polarized arrangement $(V, \eta', \xi)$ for any 
$\eta' \in \R^I/V$.  Dually, $\a$ is bounded for $(V^\bot, -\xi, -\eta')$
for any covector $\eta'$, and  
and thus $\Delta_{\a}^\gd$ is compact.

$(6)\Rightarrow(3)$: Assume that $\Delta^\gd_\a$ is compact.
Then the ring $e_\a A e_\a$, which is isomorphic to the
subring $R^\gd_{\a}$ of $B^\gd$, is {\bf Gorenstein}: there is an isomorphism
\[\int\colon (e_\a A e_\a)_{\dim \Delta^\gd_\a} \to \R\]
such that $\langle x, y\rangle = \int xy$ defines a perfect 
pairing on $e_\a A e_\a$.  If the arrangement is
rational, this can be deduced from Poincar\'e duality 
for the $\Q$-smooth toric variety $X_{\Delta^\gd_\a}$.
The general case can be deduced, for instance, from
\cite[2.5.1 \& 2.6.2]{Tim}.

We extend this pairing to a pairing
$\langle -,-\rangle: e_\a A\times Ae_\a\to \R$
by the same formula.  We claim that this is again 
a perfect pairing.  Assuming this claim, 
it defines an isomorphism $d(\Pro\a)=(e_\a A)^*\cong Ae_\a=\Pro\a$ 
of right $A$-modules, since the right and left
actions of $A$ on $e_\a A$ and $Ae_\a$ are adjoint 
under the pairing.

To prove the claim, take any non-zero element 
$x \in e_\a A e_\b$.  It will be enough to 
show that for $p = p(\b,\a)$ the map 
$\cdot p\colon e_\a A e_\b \to e_\a Ae_\a$ is injective,
since then $xp \ne 0$, which implies that there
exists $y \in e_\a Ae_\a$ so that 
$\int (xp)y = \int x(py) = \langle x, py\rangle$ is non-zero.
Using Theorem \ref{convolution equals quiver}, we need to 
show that multiplication by $1_{\b\a}$ gives an injection
from $1_{\a\a}B^\gd 1_{\b\b}$ to 
$1_{\a\a}B^\vee 1_{\a\a}$.  Following
the definition of the multiplication, we need to 
show that 
\[\cdot u_{S(\a\b\a)}\colon R^\gd _{\a\b} \to R^\gd_{\a\a}\]
is injective.  This can be deduced from the second statement of
\cite[2.4.3]{Tim}.
\end{proof}

\begin{remark}
  The equivalence $(1)\iff (6)$ is part (3) of Theorem (B), keeping in mind that $A = A(\cV) \cong
  B(\cV^\gd)$.  If $\cV$ is rational, then the set of $\a\in\cP$ for
  which $\Delta^\gd_\a$ is compact indexes the components of the core
  of the hypertoric variety $\M_{\cH^\gd}$ (Section \ref{toric}),
  which is the set of all irreducible projective lagrangian
  subvarieties of $\M_{\cH^\gd}$.
\end{remark}
  
\begin{remark}
Theorems \ref{CA-highest-weight}, \ref{Koszul dual}, and \ref{self-dual} are all
analogous to theorems that arise in the study of parabolic category
$\cO$ and other important categories in representation theory \cite{MS}.
\end{remark}

\section{Derived equivalences}\label{derived equivalences}
The purpose of this section is to show that the dependence of $A(\cV)$
on the parameters $\xi$ and $\eta$ is relatively minor.  Indeed,
suppose that $$\cV_1 = (V, \eta_1, \xi_1)\,\,\,\text{ and }\,\,\,\cV_2
= (V, \eta_2, \xi_2)$$ are polarized arrangements with the same
underlying linear subspace $V \subset \R^I$.  Thus $\cV_1$ and $\cV_2$
are related by translations of the hyperplanes and a change of
affine-linear functional on the affine space in which the hyperplanes
live. The associated quiver algebras $A(\cV_1)$ and $A(\cV_2)$ 
are not necessarily isomorphic, nor even Morita equivalent.
They are, however, {\em derived} Morita equivalent, as stated in 
Theorem (C) of the Introduction and proved in Theorem \ref{equivalence}
of this section.  That is, the triangulated category
$D(\cV)$ defined in Section \ref{Koszul algebras}
is an invariant 
of the subspace $V\subset\R^I$.  
Corresponding results for derived categories of ungraded and $dg$-modules 
can be obtained by similar reasoning.

\label{sec:derived-equivalences}
\subsection{Definition of the functors}\label{Definition of the functors}
We begin by restricting our attention to the special case in which $\xi_1 = \xi_2 = \xi$ for some $\xi \in V^*$.
On the dual side, this means that $\eta^\gd_1 = \eta^\gd_2 = \eta^\gd = -\xi$, and therefore that
the arrangements $\cH_1^\gd$ and $\cH_2^\gd$ that they define are the same; call 
this arrangement $\cH^\gd$.  Thus the sets $\cP_1 = \cP^\gd_1$ and $\cP_2 = \cP^\gd_2$
of bounded feasible chambers of $\cV_1$ and $\cV_2$ are both subsets of
the set $\cF^\gd$ of feasible chambers of this arrangement.

Our functor will be the derived tensor product with a bimodule $N$. 
We will give two equivalent descriptions of $N$, one on the A-side and one 
on the B-side, exploiting the isomorphism 
$$A_j := A(\cV_j)\cong B(\cV_j^\gd) =: B_j^\vee$$ 
of Theorem \ref{convolution equals quiver} for $j=1,2$.  
We begin with the B-side description, as it is the easier of
the two to motivate.  We define
\[N \,\,\,\, = \bigoplus_{(\a,\b)\,\in\,\cP_1\times\cP_2}\!\!\!\! R^\vee_{\a\b}[-d^\vee_{\a\b}],\]
with the natural left $B^\gd_1$-action and right $B^\gd_2$-action
given by the $\conv$ operation.

When $\cV_1$ and $\cV_2$ are rational, we have a topological description
of this module as in Section \ref{section:topological B}. 
The relative cores $$\eX^\gd_j = \bigcup_{\a\in\cP_j}X^\gd_{\a}$$
sit inside the extended core 
$$\eX^\gd_{\ext} = \bigcup_{\a\in\cF^\gd}X^\gd_{\a},$$ 
which depends only
on $\cH^\gd$ and is therefore the same for $\cV_1^\gd$ and $\cV_2^\gd$.
We then have an (ungraded) isomorphism 
$$N \,\,\,\cong\,\,\, H^*(\wt\eX^\gd_1\times_{\eX^\gd_{\ext}}\wt\eX^\gd_2)\,\,\,\, \cong
\bigoplus_{(\a,\b)\,\in\,\cP_1\times\cP_2}\!\!\!\!H^*(X^\gd_{\a\b})[-d_{\a\b}],$$
with the bimodule structure defined by the convolution operation of Section \ref{section:topological B}.

To formulate this definition on the A-side, rather than considering all feasible
sign vectors we must consider all bounded sign vectors.
Let $A_{\ext}(\cV)$ be the algebra defined by the same
relations as $A(\cV)$, but without the feasibility restrictions.
That is, we begin with a quiver $Q_{\ext}$ whose nodes are indexed
by the set $\{\pm 1\}^I$ of all sign vectors, and let $A_{\ext}(\cV)$
be the quotient of $P(Q_{\ext})\otimes_\R\Sym V^*$ by the following relations:
\begin{enumerate}
\item[\Aeone:] If $\alpha\in \{\pm 1\}^I\ssm\cB$, then $e_\alpha=0$.\\
\item[\Aetwo:] If four {\em distinct} elements $\alpha,\beta,\gamma,\delta \in \{\pm 1\}^I$
satisfy $\alpha \lra \beta \lra \gamma \lra \delta \lra \alpha$,
then 
\[p(\alpha, \beta, \gamma) = p(\alpha,\delta,\gamma).\]
\item[\Aethree:] If $\a, \b \in \{\pm 1\}^I$ and $\a\lrao i\b$,
 then
$$p(\alpha,\beta,\alpha) = t_ie_\alpha.$$
\end{enumerate}
Note that since $\cB_1 = \cB_2$, we have $A_{\ext}(\cV_1) = A_{\ext}(\cV_2)$,
which we will simply call $A_{\ext}$.
Let $$e_{\eta_j} = \sum_{\a\in\cP_j}e_\a\in A_{\ext}.$$
Then $A_j$ is isomorphic to the subalgebra $e_{\eta_j}A_{\ext}\, e_{\eta_j}$
of $A_{\ext}$.  Consider the graded vector space
$$N = e_{\eta_1}A_{\ext}\, e_{\eta_2},$$
which is a left $A_1$-module and a right $A_2$-module in the obvious way.

Recall the algebra $B_\ext(\cV)$ introduced in Section \ref{the center}.
We have $B_\ext(\cV^\vee_1) = B_\ext(\cV^\vee_2)$, which we will simply call $B_\ext^\vee$.
The following proposition is
an easy extension of Theorem
\ref{convolution equals quiver}; its proof will be left to the reader.

\begin{proposition}
The quiver algebra $A_{\ext}$ is isomorphic to the extended convolution algebra $B^\vee_\ext$.
This isomorphism, along with the isomorphisms $A_j\cong B_j^\gd$ of Theorem 
\ref{convolution equals quiver}, induces an equivalence between our two definitions of 
the bimodule $N$.
\end{proposition}

We define a functor $\Phi:D(\cV_1)\to
D(\cV_2)$ by the formula
\begin{equation*}
\Phi(M) = M\overset{L}\otimes_{A_1} N. 
\end{equation*}  
For $\a\in\cP_j$, 
let $\Pro\a^j$ and $\Ve \a^j$ denote the corresponding projective module and
standard module for $A_j$.

\begin{proposition}\label{proj-preserved}
If $\a\in \cP_1\cap\cP_2$, then $\Phi(\Pro\a^1)=\Pro\a^2$.  
\end{proposition} 
\begin{proof}
An argument analogous to that given in Proposition \ref{taut} shows that the natural map 
 \[
 	\Gamma: \Pro\a^2 = e_\a A_2
	\to e_\a A_1 \otimes _{A_1} e_{\eta_1}A_{\ext}\, e_{\eta_2} 
	= \Pro\a^1\otimes _{A_1} N= \Phi(\Pro\a^1)
\]
taking $e_\a$ to $e_\a\otimes e_{\eta_1}e_{\eta_2}$ is an isomorphism.
\end{proof}

\begin{remark}
Note that by Proposition \ref{proj-preserved} and the equivalence 
$(3)\iff (6)$ of Theorem \ref{self-dual}, 
$\Phi$ takes self-dual projectives to self-dual projectives.
\end{remark}

Consider a basis $b\in \Bas(\cV_1) = \Bas(\cV_2)$, 
and recall that we have bijections
$$\cP_1 \overset{\mu_1}{\longleftarrow} \Bas(\cV_1) = \Bas(\cV_2) \overset{\mu_2}{\longrightarrow} \cP_2.$$
Let $\nu:\cP_1\to \cP_2$ denote the composition.
Recall also that the sets
$\cB_b\subset \cB$, defined in Section \ref{sec:partial order}, do not depend on $\eta$.

\begin{lemma}\label{standard-filter}
For any $\a\in\cP_1$, the $A_2$-module $\Phi(\Pro\a^1)$ has a
filtration with standard subquotients.  
If $\a\in\cB_b$ then
the standard module $\Ve{\mu_2(b)}^2$ appears with multiplicity 1 
in the associated graded,
and otherwise it does not appear.
\end{lemma}

\begin{proof}
As in the proof of Proposition \ref{proj-preserved}, we have
$\Phi(\Pro\a^1) = e_\a A_1 \otimes _{A_1} e_{\eta_1}A_{\ext}\, e_{\eta_2}$,
thus we may represent an element of $\Phi(\Pro\a^1)$ by a path in $\cB$
that begins at $\a$ and ends at an element of $\cP_2 = \cB\cap \cF_2$.
For $\b\in\cP_2$, let $\Phi(\Pro \a^1)_\b$ be the submodule generated
by those paths $p$ such that $\b$ is the maximal element of $\cP_2$
through which $p$ passes,
and let
$$\Phi(\Pro \a^1)_{ > \b} = \bigcup_{\ga>\b}\Phi(\Pro \a^1)_\ga\,\,\,\,\,\text{and}\,\,\,\,\,
\Phi(\Pro \a^1)_{\geq\b} = \bigcup_{\ga\geq\b}\Phi(\Pro \a^1)_\ga.$$
We then obtain a filtration
$$\Phi(\Pro\a^1) = \bigcup_\b \,\Phi(\Pro \a^1)_{\geq\b}.$$ 
Suppose that $\b = \mu_2(b)$;
we claim that the quotient
$\Phi(\Pro \a^1)_{\geq\b} \big{/} \Phi(\Pro \a^1)_{ > \b}$ 
is isomorphic to $\Ve\b^2$ if $\a\in\cB_b$,
and is trivial otherwise.

If $\a\in\cB_b$, then we have a map
$$\Ve \b^2 \rightarrow\Phi(\Pro \a^1)_{\geq\b} \big{/}  \Phi(\Pro \a^1)_{ > \b}$$
given by pre-composition with any taut path from $\a$ to $\b$, and an adaptation
of the proof of Theorem \ref{CA-highest-weight} shows that it is an isomorphism.
If $\a\notin\cB_b$, then there exists $i\in b$ such that $\a(i)\neq\b(i)$,
and any path from $\a$ to $\b$ will be equivalent to one that passes through $\b^i>\b$.
Thus in this case the quotient is trivial.
\end{proof}

\begin{proposition}\label{gg}
For all $\a\in\cP_1$, we have $[\Phi(\Ve \a^1)]=[\Ve{\nu(\a)}^2]$
in the Grothendieck group of (ungraded) right $A_2$-modules.  
Thus $\Phi$ induces an isomorphism of Grothendieck groups.  
\end{proposition}

\begin{proof}
For all $b\in\Bas$, we have
\begin{equation*}
\sum_{\a\in\cB_b}[\Phi(\Ve\a^1)]=[\Phi(\Pro{\mu_1(b)}^1)]=\sum_{\a\in \cB_b}[\Ve{\nu(\a)}^2],
\end{equation*}
where the first equality follows from the proof of Theorem \ref{CA-highest-weight}
and the second follows from Lemma \ref{standard-filter}.
The first statement of the theorem then follows from induction on $b$.
The second statement follows from the fact that the Grothendieck group
of modules over a quasi-hereditary algebra 
is freely generated by the classes of the standard modules.
\end{proof}

\begin{remark}
 We emphasize that $\Phi(\Ve \b^1)$ and $\Ve{\nu(\b)}^2$ are
 {\em not} isomorphic as modules; Proposition \ref{gg} says only that they
 have the same class in the Grothendieck group.
In fact, the next proposition provides an explicit description of $\Phi(\Ve\b^1)$
as a module.
\end{remark}

\begin{proposition}
$\Phi(\Ve\a^1)$ is the quotient of $\Phi(\Pro\a^1)$ by the submodule generated by all paths which cross the hyperplane $H_i$ for some $i\in \mu_1^{-1}(\a)$. In particular, $\mathrm{Tor}_k^{A_1}(\Ve\a^1,N)=0$ for all $k>0$.
\end{proposition}

\begin{proof}
It is clear that if we take a projective resolution of $\Ve\a^1$ and tensor it with $N$,
the degree zero cohomology of the resulting complex will be this quotient.
Thus we need only show that the complex is exact in positive degree, that is, that
it is a resolution of $\Ve\a^1\otimes N$.  The proof of this fact is identical to the proof of
Lemma \ref{standard-filter}.
\end{proof}

\begin{corollary}\label{Tor-vanish}
If a right $A_1$-module $M$ admits a filtration by standard modules, then 
$\mathrm{Tor}_k^{A_1}(M,N)=0$ for all $k>0$, and thus
$\Phi(M)=M\otimes_{A_1} N$.
\end{corollary}

\begin{remark}
Though we will not need this fact, it is interesting to note that $\Phi$
takes the exceptional collection $\{V_\a^1\}$ to the mutation of $\{V_{\nu(\a)}^2\}$
with respect to a linear refinement of our partial order.
(See \cite{Bez} for definitions of exceptional collections and mutations.)
We leave the proof as an exercise to the reader.
\end{remark}

\subsection{Ringel duality and Serre functors}\label{Ringel duality and Serre functors}
In this section we pass to an even further special case; we still
require that $\xi_1 = \xi_2$, and we will now assume in addition that
$\eta_1 = -\eta_2$.  Rather than referring to $\cV_1$ and $\cV_2$, we
will write $$\cV = (V, \eta, \xi)\,\,\, \text{and}\,\,\, \bar\cV = (V,
-\eta, \xi),$$ and we will refer to $\bar\cV$ as the {\bf reverse} of
$\cV$.  Let $A = A(\cV)$, $\bar A = A(\bar\cV)$, and let
$$\Phi^-:D(\cV)\to D(\bar\cV)\,\,\, \text{and}\,\,\, \Phi^+:D(\bar\cV)\to D(\cV)$$
be the functors constructed above. 

\begin{theorem}\label{ringel-dual}
The algebras $A$ and $\bar A$ are Ringel dual, 
and the Ringel duality functor is
$\du\circ\,\Phi^-=\Phi^-\!\circ \du$.
In particular,  $\Phi^{-}$ sends projectives to tiltings, tiltings to injectives, and standards to costandards.
\end{theorem}
 
\begin{proof}
Using the B-side description of the functor $\Phi^-$, we find that for any $\a\in\cP$,
$$\Phi^-(\Pro\a) = \bigoplus_{\bar\b\in\bar\cP}R^\gd_{\a\bar\b}[-d_{\a\bar\b}].$$
The polyhedron $\Delta^\vee_{\a\bar\b}$ is always compact, thus
$R^\vee_{\a\bar\b}$ is Gorenstein and $\Phi^-(\Pro\a)$ is self-dual.
We showed in Lemma \ref{standard-filter} that $\Phi^-(\Pro\a)$ admits a filtration
with standard subquotients, with $\bar V_{\nu(\a)}$ as its largest standard submodule,
from which we can conclude that $\Phi^-(\Pro\a)$ is isomorphic to $\bar T_{\nu(\a)}$.
Thus $\du\circ\,\Phi^-$ is a contravariant functor that sends projective
modules to tilting modules; by Proposition \ref{ringdu},
it will now be sufficient to show that $\Phi^-$ is an equivalence.

For all $\a,\b\in\cP$, the functor $\Phi^-$ induces a map $\Hom(\Pro\a,\Pro\b)\to\Hom(\bar
T_{\nu(\a)},\bar T_{\nu(\b)})$.  
We will show that this map is an isomorphism by first showing it to be injective and then
comparing dimensions.
By the double centralizer property (Remark \ref{dcp}),
there exists a self-dual projective $P_{\a'}$ and a
map $\Pro{\a'}\to \Pro\a$ such
that composition with this map defines an injection from $\Hom(\Pro\a,\Pro\b)$
to $\Hom(\Pro{\a'},\Pro\b)$.  On the other hand, the
injective hull of $\Pro\b$ is the same as the injective hull of its
socle.  Since $P_\b$ has a standard filtration, each simple summand of this socle 
lies in the socle of some standard module.  Then the implication 
$(4)\impl(3)$ of Theorem \ref{self-dual} tells us that the injective hull of $P_\b$
is isomorphic to some self-dual projective $P_{\b'}$.

Now consider the commutative diagram below, in which the
vertical arrow on the left is injective.
\[
\xymatrix{ 
\Hom(\Pro\a,\Pro\b)\ar[rr] \ar[d]&& \Hom(\bar
T_{\nu(\a)},\bar T_{\nu(\b)}) \ar[d]  \\
\Hom(\Pro{\a'},\Pro{\b'})\ar[rr] && \Hom(\bar
T_{\nu(\a')},\bar T_{\nu(\b')})
}
\]
To prove injectivity of the top horizontal
arrow, it is enough to show injectivity of the bottom horizontal arrow, which follows 
from Proposition \ref{proj-preserved}.

Next we need to prove that the two Hom-spaces on the top of the diagram
have the same dimension.
Since standards and costandards are dual sequences 
(Proposition \ref{dual sequences}), we have 
$$\Ext^i( T_{\a}, T_{\b}) =0\,\,\,\text{for all $i>0$}.$$

By Lemma \ref{standard-filter}, we have the decomposition 
\begin{equation*}
  [\bar T_{\nu(\a)}] = [\Phi^-(P_\a)]=\sum_{\a\in\cB_b}[V_{\mu_2(b)}]
\end{equation*}
in the Grothendieck group of $\bar A$-modules.
From this statement and Proposition \ref{gg}, we may deduce that
$$[P_\a] = \sum_{\a\in\cB_b}[V_{\mu_1(b)}]$$ in the Grothendieck
group of $A$-modules.
The standard classes form an orthonormal basis with respect to the Euler form 
(Remark \ref{consequences}), thus
\begin{align*}
  \dim\Hom(\bar T_{\nu(\a)}, \bar T_{\nu(\b)})&=
  \big\langle [\bar T_{\nu(\a)}],[\bar T_{\nu(\b)}]\big\rangle\\
  &=\#\{b\in\Bas\mid\a,\b\in\cB_b\}\\
  &=\big\langle [P_\a],[P_\b]\big\rangle\\
  &=\dim\Hom(P_{\a},\bar  P_{\b}).
\end{align*}
Thus $\Phi^-$ is an equivalence of categories.

By Propositions \ref{ringdu} and \ref{duality}, it is now sufficient to show that
$R(A)$ is isomorphic to $\bar A$.  To this end, consider the equivalence
$\Phi^+$ from $\bar A$ modules to $A$ modules, which takes $\bar P_{\nu(\a)}$
to $T_\a$ for all $\a\in\cP$.  From this we find that
$$R(A) = \End_A(\oplus\Tilt\a)\cong \End_{\bar A}(\oplus\bar P_{\nu(\a)}) = \bar A.$$
The last statement follows from Proposition \ref{duality}.
\end{proof}

The functors $\Phi^\pm$ are not mutually inverse (we will see this
explicitly in Theorem~\ref{serre-functor}), but their composition is
interesting and natural from a categorical perspective.  For any
graded algebra $E$, an auto-equivalence $\mathbb{S}:D(E)\to D(E)$
is called a {\bf Serre functor}\footnote{This terminology is of course motivated by
Serre duality on a projective variety.} if we have isomorphisms of
vector spaces
\begin{equation*}
\Hom(M,\mathbb{S}M')\cong \Hom(M',M)^*
\end{equation*}
that are natural in both $M$ and $M'$.

By the 5-lemma, to check that a functor is Serre, we need only show that it
is exact and satisfies the Serre property on homomorphisms
between objects in a generating set of the category. If $E$ is finite-dimensional
and has finite global dimension, then $D(E)$ is generated by $E$ as a module 
under right multiplication, so an exact functor $\mathbb{S}:D(E)\to D(E)$ is Serre if and only if
$$\mathbb{S}E\cong \Hom(E,\mathbb{S}(E)) \cong \Hom(E,E)^* \cong E^*.$$
Since every right $E$-module has a free resolution,
any Serre functor is equivalent to the derived tensor product with $\mathbb{S}(E)$,
hence $-\overset{L}\otimes_E E^*$ is the unique Serre functor on $D(E)$.
It follows that $E$ is symmetric in the sense of Definition \ref{symdef} if and
only if its Serre functor is trivial.

\begin{theorem}\label{serre-functor}
The endofunctor $\Phi^+\circ\Phi^-$ is a Serre functor of $D(\cV)$.
\end{theorem}
\begin{proof}
We use the characterization of \cite[3.4]{MS}: $\mathbb{S}$ is a Serre functor if and only if 
\begin{enumerate}
\item $\mathbb{S}$ sends projectives to injectives, and
\item $\mathbb{S}$ agrees with the Serre functor of $\PIS$ (Equation \eqref{erect S}
of Section \ref{sdp})
on the subcategory of projective-injective modules.
\end{enumerate}
Condition (1) follows immediately from Theorem~\ref{ringel-dual}, since $\Phi^-$ sends projectives to tiltings, which $\Phi^+$ (by symmetry) sends to injectives.
Since $\PIS$ is symmetric (Definition \ref{symdef}), its Serre functor is trivial, and condition (2) says simply that $\mathbb{S}$ 
must act trivially on projective-injective modules.  This follows from Theorem \ref{self-dual} and Proposition~\ref{proj-preserved}.
\end{proof}

\subsection{Composing functors}
We now return to the situation of Section \ref{Definition of the functors},
in which we have two polarized arrangements
$$\cV_1 = (V, \eta_1, \xi)\,\,\,\text{ and }\,\,\,\cV_2= (V, \eta_2, \xi).$$
To this mix we add a third polarized arrangement
$\cV_3 = (V, \eta_3, \xi),$
and study the composition of the two functors
$$D(\cV_1)\overset{\Phi_{12}}{\longrightarrow}
D(\cV_2)\overset{\Phi_{23}}{\longrightarrow}
D(\cV_3).$$
Since
$\Phi_{12}$ and $\Phi_{23}$ are the derived functors of tensor product with a bimodule, their
composition is the derived functor of the derived tensor product of
these bimodules.  It is an easy exercise to check that the right 
$A_2$-module $N_{12}$ admits a standard filtration,
hence Corollary~\ref{Tor-vanish} tells us that the higher levels of the derived
tensor product of $N_{12}$ and $N_{23}$ vanish.
Thus for any right $A_1$-module $M$, we have
\begin{equation*}
\Phi_{23}\circ\Phi_{12}(M)=(M\overset{L}\otimes_{A_1} N_{12})\overset{L}\otimes_{A_2} N_{23} =M\overset{L}\otimes_{A_1} (N_{12}\otimes_{A_2} N_{23}).
\end{equation*}

There a natural map $N_{12}\otimes_{A_2}N_{23}\to N_{13}$ given by composition of paths,
which induces a natural transformation
$\Phi_{23}\circ\Phi_{12}\to\Phi_{13}$.
Furthermore, Proposition \ref{gg} tells us that $\Phi_{23}\circ\Phi_{12}$ and $\Phi_{13}$
induce the same map on Grothendieck groups.  In particular, this implies
that the bimodules  $N_{12}\otimes_{A_2}N_{23}$ and $ N_{13}$ have the same dimension.

Suppose that $\eta_3 = -\eta_1$, so that $\cV_3 = \bar\cV_1$,
$$\Phi_{23} = \Pto:D(\cV_2)\to D(\bar\cV_1),\,\,\,\text{and}\,\,\, 
\Phi_{13} = \Phi^-:D(\cV_1)\to D(\bar\cV_1).$$
Lemma \ref{compose-to-reverse} says that, in this case,
the natural transformation from $\Phi_{23}\circ\Phi_{12}$ to $\Phi_{13}$
is an isomorphism.

\begin{lemma}\label{compose-to-reverse}
$\Phi^-\cong\Pto\circ\Pot$
\end{lemma}
\begin{proof}
We would like to show that the natural map $N_{12}\otimes_{A_2}N_{2\bar 1}\to N_{1\bar 1}$
is an isomorphism.  We have already observed that the source and target have the same dimension,
so it is enough to show that the map is surjective.  In other words, we must show that for any
$\a\in\cP_1$ and $\b\in\bar\cP_1$, every element of $e_\a A_{\ext}e_\b$ may be represented by a path in $Q_{\ext}$ that passes through a node in $\cP_2$.

The existence of a non-zero element of $e_\a A_{\ext}e_\b$ is equivalent to
both sign vectors remaining bounded if the set $S$ of hyperplanes
separating them is deleted.  
Thus we may assume that $\a|_{I\ssm S} = \b|_{I\ssm S}$ is bounded feasible
for both $(\cV_1)_S$ and its reversal $(\bar\cV_1)_S = \overline{(\cV_1)_S}$.
But this implies that the same sign vector is bounded feasible for $(\cV_2)_S$,
thus there must exist a sign vector $\ga\in\cP_2$ such that $\ga|_{I\ssm S} = \a|_{I\ssm S} = \b|_{I\ssm S}$.
Then by Corollary~\ref{visits every region}, our
element can be written as a sum of paths passing through the node $\ga$.
\end{proof}

This allows us to prove the main theorem of Section \ref{derived equivalences}.

\begin{theorem}\label{equivalence}
The categories $D(\cV_1)$ and $D(\cV_2)$ are equivalent.
\end{theorem}

\begin{proof}
We first note that by Theorems \ref{koszul-equiv} and \ref{Koszul dual}, we have equivalences 
$$D(\cV_1)\simeq D(\cV_1^\gd)\,\,\,\text{ and }\,\,\,
D(\cV_2) \cong D(\cV_2^\gd).$$
Since $\eta_j^\gd = -\xi_j$,
replacing the parameter $\xi_1$ with $\xi_2$ can be interpreted on the Gale 
dual side as replacing the parameter $\eta^\gd_1$ with $\eta^\gd_2$. 
Thus we may reduce Theorem \ref{equivalence} 
to the case where $\xi_1$ and $\xi_2$ coincide. 

By Lemma~\ref{compose-to-reverse}, $\Pto\circ\Pot=\Phi^-$, which we know from Theorem~\ref{ringel-dual} is an equivalence of derived categories.  Thus $\Pot$ is faithful and $\Pto$ is full and essentially surjective. 
By symmetry, $\Pot$ is also full and essentially surjective, thus an equivalence. 
\end{proof}

\bibliography{./symplectic}
\bibliographystyle{amsalpha}
\end{spacing}

\end{document}